\documentclass[3p]{elsarticle}
\usepackage[utf8]{inputenc} \usepackage[T1]{fontenc}
\pdfoutput=1

\reversemarginpar

\usepackage{graphicx}
\usepackage{amsmath,amssymb}
\usepackage[amsmath,amsthm,thmmarks]{ntheorem}
\usepackage{mathtools}
\usepackage{mathrsfs}
\usepackage{textcomp}
\usepackage{svg}
\usepackage[final]{pdfpages}
\usepackage{graphicx}
\usepackage{tikz}
\usepackage{tkz-graph}
\usetikzlibrary{patterns,angles,quotes,arrows,shapes,calc,positioning}
\usepackage[europeanresistors,europeaninductors,siunitx,nooldvoltagedirection]{circuitikz}

\theoremstyle{plain}
\newtheorem{definition}{Definition}[section]

\newtheorem{proposition}[definition]{Proposition}
\theoremstyle{definition}
\newtheorem{remark}[definition]{Remark}
\newtheorem{example}[definition]{Example}

\newcommand{\Lc}{\mathcal{L}}

\newcommand{\Ec}{\mathcal{E}}
\newcommand{\Fc}{\mathcal{F}}

\newcommand{\Rc}{\mathcal{R}}
\newcommand{\Pc}{\mathcal{P}}

\newcommand{\rk}{{\rm rk\,}}

\newcommand{\N}{{\mathbb{N}}}
\newcommand{\R}{{\mathbb{R}}}

\DeclareMathOperator{\im}{im}

\DeclareMathOperator{\init}{init}
\DeclareMathOperator{\ter}{ter}

\newcommand{\ddt}{{\textstyle\frac{d}{dt}}}

\DeclareMathAlphabet{\mathcal}{OMS}{cmsy}{m}{n}
\DeclareMathAlphabet{\mathpzc}{OT1}{pzc}{m}{it}

\newcommand{\cT}{\scalebox{.85}{$\mathpzc{T}$}}

\newcommand{\cR}{\scalebox{.85}{$\mathpzc{R}$}}

\newcommand{\cCl}{\scalebox{1.27}{$\mathpzc{C}$}}
\newcommand{\cC}{\scalebox{.85}{$\mathpzc{C}$}}

\newcommand{\cD}{\scalebox{.85}{$\mathpzc{D}$}}
\newcommand{\cDs}{\scalebox{.7}{$\mathpzc{D}$}}

\newcommand{\cI}{\scalebox{.85}{$\mathpzc{I}$}}

\newcommand{\cJ}{\scalebox{.7}{$\mathpzc{J}$}}

\newcommand{\cLl}{\scalebox{1.27}{$\mathpzc{L}$}}
\newcommand{\cL}{\scalebox{.85}{$\mathpzc{L}$}}

\newcommand{\cN}{\scalebox{.85}{$\mathpzc{N}$}}

\newcommand{\cO}{\scalebox{.85}{$\mathpzc{O}$}}

\newcommand{\cS}{\scalebox{.85}{$\mathpzc{S}$}}

\newcommand{\cV}{\scalebox{.85}{$\mathpzc{V}$}}

\biboptions{sort&compress}

\begin{document}

\date{\today}
\journal{Journal of Geometry and Physics}

\begin{frontmatter}



\title{Port-Hamiltonian formulation of nonlinear electrical circuits}


\author[1]{H. Gernandt}\ead{hannes.gernandt@tu-ilmenau.de}% \fnref{fn1}}
\author[2]{F. E. Haller} \ead{frederic.haller@uni-hamburg.de}%\corref{cor1}
\author[2]{T. Reis} \ead{timo.reis@uni-hamburg.de}
\author[4]{A. J. van der Schaft} \ead{a.j.van.der.schaft@rug.nl}

\address[1]{TU Ilmenau, Weimarer Straße 25, 98693 Ilmenau, Germany} \address[2]{Universität Hamburg, Bundesstraße 55, 20146 Hamburg, Germany
}%\address[3]{Universität Hamburg, Bundesstraße 55, 20146 Hamburg Germany}
\address[4]{Bernoulli Institute for Mathematics, CS and AI, University of Groningen, The Netherlands}

\begin{abstract}
We consider nonlinear electrical circuits for which we derive a port-Hamiltonian formulation. After recalling a framework for nonlinear port-Hamiltonian systems, we model each circuit component as an individual port-Hamiltonian system. The overall circuit model is then derived by considering a port-Hamiltonian interconnection of the components. We further compare this modelling approach with standard formulations of nonlinear electrical circuits.
\end{abstract}


% \begin{highlights}
% \end{highlights}


\end{frontmatter}

\section{Introduction}
\noindent Port-Hamiltonian system models encompass a very large class of nonlinear physical systems \cite{JvdS14,vdS17}  and arise from port-based network modelling of complex lumped parameter systems from various physical domains, such as, for instance, mechanical and electrical systems. Modelling by port-Hamiltonian systems has gained a lot of attention, see, for instance, the surveys \cite{JvdS14, vdS13}. Tremendous progress has been recently made in port-Hamiltonian modelling of constrained dynamical systems, which leads to differential-algebraic equations \cite{BMXZ18,MMW18,MvdS18,MvdS19,vdS13}. This enables to apply the framework to modelling of multibody systems with holonomic and non-holonomic constraints as well as electrical circuits. Examples of the latter class has been considered from a~port-Hamiltonian point of view in \cite{vdS10,JvdS14,vdS13, VvdS10a}. However, an approach to electrical circuits has been only made for the case where the circuit contains only capacitances and inductances \cite{BMvdS95}. The recent progress in port-Hamiltonian differential-algebraic equations however allows to treat a~by far wider class of electrical circuits. This is exactly the purpose of this article, where we consider a~variety of electrical components, such as resistances, capacitances, inductances, diodes, transformers, transistors, current sources and voltage sources from a~port-Hamiltonian perspective. Thereafter, we consider the circuit interconnection structure by utilizing the underlying graph of the given electrical circuit. This gives rise to a port-Hamiltonian model, which only incorporates the Kirchhoff laws. Finally, the port-Hamiltonian model of the electrical circuit is obtained by an interconnection with the individual port-Hamiltonian systems representing the components.\\
We will compare the resulting dynamical system with well-known formulations of nonlinear electrical circuits like the (charge/flux-oriented) \emph{modified nodal analysis} and the \emph{modified loop analysis}.

%
%
% This encompasses for example so-called RLC circuits, i.e., circuits composed of resistances, capacitances and inductances. Additionally, our approach allows for current and voltage sources and also \emph{multi-port}\footnote{In the sense of electrical components; see also the discussion in \cite{Wil10}} elements such as transformers. The fundamental idea behind the port-Hamiltonian modelling of electrical circuits is the following. First, model each component of the electrical circuit as individual port-Hamiltonian systems. Next

%  We then follow the approach of \cite{VvdS10b} by interconnecting the individual components by regarding this Dirac structure as an \emph{interconnection Dirac structure}.
% We brievly discuss the correspondence between inpout/output and flow/effort with respect to circuits in light of a different notion of port-Hamiltonian system presented in \cite{MMW18}.

\section{Port-Hamiltonian systems and their interconnections}
\subsection{Port-Hamiltonian DAE systems}
\noindent We review some basics in port-Hamiltonian differential-algebraic equations (DAEs) from \cite{MvdS18,MvdS19}. An important concept is that of the \emph{Dirac structure}, which describes the power preserving energy-routing of the system. In a very general setting, a Dirac structure on a manifold
%\footnote{All notions will be assumed to be smooth; i.e., $C^\infty$.}
 $\mathcal M$ is defined \cite[Def.~2.2.1]{Cou90} as a {certain} subbundle of $\mathcal D\subset T\mathcal M\oplus T^*\mathcal M$ (i.e., the direct sum of the tangent bundle and co-tangent bundle of $\mathcal{M}$). It turns out that, even for nonlinear circuits, this general definition is not needed, and we may introduce Dirac structures only for the simple case where $\mathcal M=\R^n$ (which gives rise to the identification $T^*\R^n\cong T\R^n\cong\R^n\times\R^n$) and $\mathcal D\subset \R^n\times\R^n$ is a~subspace.% This ,
%for which we give the following definition from \cite[Sec.~3]{Dor87}.

\begin{definition}[Dirac structure]\label{def-Dir}%[\!{\cite[Definition 1.1]{Cou90}\cite[3. Dirac structure]{Dor87}}]
A subspace $\mathcal D \subset \R^n\times \R^n$ is called a \emph{Dirac structure}, if
for all $f,e\in\R^n$ holds
\begin{equation*}%\label{lag-cond}
(f,e)\in \mathcal D\;\Longleftrightarrow\; \forall \, (\hat{f},\hat{e})\in \mathcal D:\; e^\top \hat{f}+\hat{e}^\top f=0.
\end{equation*}
\end{definition}
We will also write $(f,e)\in \mathcal D\subset \mathcal F\times\mathcal E$, where $\mathcal F$ denotes the space of \emph{flows} and $\mathcal E=\R^n\cong\mathcal F^*$ denotes the space of \emph{efforts}. A useful characterisation of Dirac structures is the following.
\begin{proposition}[\!\! {\cite[Prop.~1.1.5]{Cou90}}]\label{prop-kernel}
A subspace $\mathcal D\subset \R^n\times\R^n$ is a Dirac structure if, and only if, there exist $K,L\in\R^{n\times n}$ with $KL^\top+LK^\top=0$ and $\rk[K\;\;L]=n$, such that
\begin{equation}\label{eq:kernel}
\mathcal D=\left\{(f,e)\in\R^n\times\R^n~\vert~ Kf+Le=0\right\}.
\end{equation}
\end{proposition}
Now we introduce a~relation describing the energy storage of the system and is called \emph{Lagrange submanifold}. Again, the general definition of Lagrange submanifold as found in {\cite[p.\ 568]{Lee12}} is not needed for nonlinear circuits. It suffices to consider the case of submanifolds  of $\R^{n}\times \R^{n}$. Typically, the manifolds are assumed to be smooth. This can however be relaxed, and we may consider less-smooth manifolds for our purposes.
\begin{definition}[Lagrange submanifold]\label{def-Lag}
A submanifold $\mathcal L\subset \R^{n}\times \R^{n}$ is called \emph{Lagrange submanifold} of $\R^{n}\times\R^{n}$, if for all $x\in \mathcal L$ and ${(v_1,v_2)}\in\R^{n}\times \R^{n}$ holds
\begin{equation*}%\label{lag-cond}
{(v_1,v_2)}\in T_x\mathcal L\;\Longleftrightarrow\; \forall {(w_1,w_2)}\in T_x\mathcal L: {v_1^\top w_2-v_2^\top w_1}=0.
\end{equation*}
Hereby, $T_x\mathcal{L}\subset\R^{n}\times\R^{n}$ stands for tangent space of $\mathcal{L}$ at $x\in\mathcal{L}$.
\end{definition}
In the following we show that gradient fields induce Lagrange submanifolds.
\begin{proposition}\label{prop-gradient}
Let $Q:\R^n\to\R^n$ be continuously differentiable. Then the submanifold consisting of the graph of $Q$, i.e.,
\[\mathcal L_Q\coloneqq\{(x,Q(x))\in \R^n\times\R^n~|~x\in\R^n\}\]
is a~Lagrange submanifold if, and only if, $Q$ is a~gradient field. In other words,
there exists some twice continuously differentiable function $H:\R^n\to\R$ such that $\nabla H=Q$.
\end{proposition}
\begin{proof}
Using that $\R^n$ is simply connected, the case of smooth $Q$ follows from \cite[Prop.~22.12]{Lee12}. The less smooth case follows by a~straightforward modification of the proof of \cite[Prop.~22.12]{Lee12}.
\end{proof}
%Although we require (sub-)manifolds to be smooth in our framework, Proposition \ref{prop-gradient} shows us that we also can regard certain less smooth submanifolds as Lagrange manifolds.
The case where a~Lagrangian submanifold is a~subspace deserves special attention.
\begin{proposition}[\!\! {\cite[Prop.~5.2]{MvdS18}}]\label{prop-kernel_L}
A subspace $\mathcal L\subset \R^n\times\R^n$ is a~Lagrangian submanifold if, and only if,
\[\mathcal L=\left\{(f,e)\in\R^n\times\R^n~\vert~ S^\top f=P^\top e\right\}\]
for some matrices $S,P\in\R^{n\times n}$ with $S^\top P=P^\top S$ and $\rk[S^\top\,P^\top]=n$.
\end{proposition}
Another concept needed for port-Hamiltonian systems is that of the \emph{resistive relation}, which represents the internal energy dissipation of the system. It is defined as a relation on the \emph{space of resistive flows} $\mathcal F_R$ and \emph{space of resistive efforts} $\mathcal E_R$ \cite[Sec.~2.4]{JvdS14}. In our setting, both $\mathcal E_R$ and $\mathcal F_R$ will be again $\R^n$.

\begin{definition}[Resistive relation]\label{def-res}
A relation $\mathcal R\subset\R^n\times\R^n$ is called \emph{resistive}, if
\[\forall~(f_\Rc,e_\Rc)\in\Rc:\ e_\Rc^\top f_\Rc\leq0.\]
\end{definition}
%\begin{remark}
%In the general setting, when having a Dirac structure of the form $\mathcal D\subset T\mathcal M\oplus T^*\mathcal M$ with $\mathcal D(x)\subset T_x\mathcal M\times T_x^*\mathcal M$, one also speaks of \emph{modulated Dirac structure} (\cite{JvdS14}) because it pointwise defines (modulated by $x$) a Dirac structure corresponding to our definition, which in turn, is accordingly also referred to as \emph{constant Dirac structure} (\cite{MvdS19}). Does a similar terminology for Lagrange submanifolds make sense? A \emph{Lagrange structure} $L\subset \R^n\times\R^n$ is defined (\cite{MvdS18}) as a subspace satisfying
%\[L=\left\{(v,w)\in\R^n\times\R^n~\vert~\forall (v',w')\in L: w^\top v'-w'^\top v=0\right\}.\]
%With the identifications (\ref{ident}), we see that $\omega_x((v,w),(v',w'))=w'^\top v-w^\top v'$, i.e., $T_x\mathcal L$ is a Lagrange structure. Thus by abuse of language, $T\mathcal L$ pointwise  defines a Lagrange structure. We point out that the resistive structure may also be modulated in a more general setting (\cite{JvdS14}).
%\end{remark}
Having defined Dirac structures, Lagrange submanifolds and resistive relations, we are now ready to introduce port-Hamiltonian systems. Again note this class can be defined in a~more general setting by using manifolds \cite{MvdS19,JvdS14}. We `boil this down' to the setup which is needed for electrical circuits.
\begin{definition}[Port-Hamiltonian (pH) system]\label{def-pH}
Let $n_\Lc,n_\Rc,n_\Pc\in\N_0$ and denote \[\mathcal F_{\Lc}=\mathcal E_{\Lc}=\R^{n_\Lc},\quad\mathcal F_\Rc=\mathcal E_\Rc=\R^{n_\Rc},\quad\mathcal F_\Pc=\mathcal E_\Pc=\R^{n_\Pc}.\]
 A \emph{port-Hamiltonian (pH) system}  is a triple $(\mathcal D,\mathcal L,\mathcal R)$, where $\mathcal D\subset (\mathcal F_{\mathcal L}\times\mathcal F_R\times\mathcal F_P)\times(\mathcal E_{\mathcal L}\times\mathcal E_R\times\mathcal E_P)$ is a~Dirac structure (see Definition~\ref{def-Dir}), $\mathcal L\subset \mathcal F_{\mathcal L}\times \mathcal E_{\mathcal L}$ is a~Lagrange submanifold (see Definition~\ref{def-Lag}) and $\mathcal R\subset\mathcal F_\Rc\times\mathcal E_\Rc$ a~resistive relation (see Definition~\ref{def-res}).\\
The elements of $\mathcal F_{\mathcal L}$, $\mathcal E_{\mathcal L}$, $\mathcal F_\Rc$, $\mathcal E_\Rc$, $\mathcal F_\Pc$, $\mathcal E_\Pc$ are, accordingly, called the \emph{energy-storing flows/efforts}, \emph{resistive flows/efforts} and  \emph{external flows/efforts}.\\
The \emph{dynamics} of the pH system are specified by the differential inclusion
\[
\begin{aligned}
 (-\ddt x(t),f_\Rc(t),f_\Pc(t),e_\Lc(t),e_\Rc(t),e_\Pc(t))\in\mathcal D,\ (x(t),e_\Lc(t))\in\mathcal L,\ (f_\Rc(t),e_\Rc(t))\in\mathcal R.
\end{aligned}
\]
\end{definition}
%\begin{definition}[Port-Hamiltonian system]%\label{def-pH}
%Let $n_S,n_R,n_P\in\N_0$. A \emph{port-Hamiltonian (pH) system}  is a triple $(\mathcal D,\mathcal L,\mathcal R)$ consisting of a~Dirac structure
%$\mathcal D\subset \R^{n_S+n_R+n_P}\times\R^{n_S+n_R+n_P}$,
%a Lagrange submanifold
%$\mathcal L \subset \R^{n_S}\times \R^{n_S}$,
%and a resistive structure
%$\mathcal R\subset \R^{n_R}\times\R^{n_R}$.\\
%The dynamics  are specified by the differential equation
%\[
%\begin{aligned}
%\forall t\in\R:\ (-\ddt x(t),f_R(t),f_P(t),e_S(t),e_R(t),e_P(t))\in\mathcal D,\; (x(t),e_S(t))\in\mathcal L,\; (f_R(t),e_R(t))\in\mathcal R.
%\end{aligned}
%\]
%\end{definition}
Note that, in this paper, we do not investigate any solvability theory of the resulting equations.
% and therefore do not mention any properties, e.g. which kind of differentiability, the functions involved in the dynamics should have.

\begin{figure}
	\centering
	\includegraphics[width=0.4\textwidth]{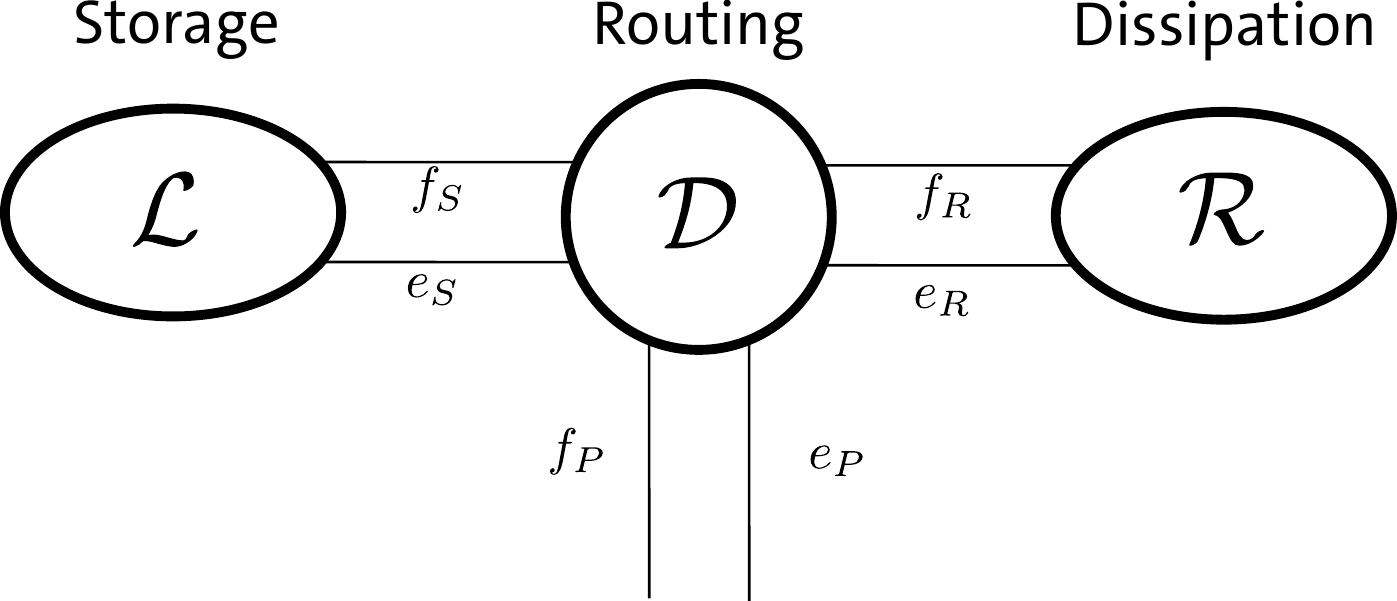}
	\caption{Visual representation of a pH system.}
	% \label{pH-sys}
\end{figure}

\subsection{Interconnection of port-Hamiltonian systems}
\noindent A key property of pH systems is that this class is closed under power-conserving interconnection. Different methods of how to design such interconnections are for example elucidated in \cite{BCGM18,CvdSB07,JvdS14,VvdS10b}. The interconnection we will be using for the electrical circuits follows the ideas presented in \cite{JvdS14}.
Interconnection is based on the assumption that each system has two kinds of external flows and efforts, namely specific and to-be-linked ones, where the latter ones are belonging to the same space for each Dirac structure.
\begin{definition}[Interconnection of pH systems]\label{def:pHint}
For $i=1,2$, let $(\mathcal D_i,\mathcal L_i,\mathcal R_i)$ be two pH systems with specific flow and effort spaces,
\[\begin{aligned}
\mathcal F_i=&\, \mathcal F_{\mathcal{L}i}\times \mathcal F_{\Rc i}\times \mathcal F_{\Pc i}\times\mathcal F_{\Pc{\rm link}},\quad \mathcal E_i=\mathcal F_{\mathcal{L}i}\times\mathcal E_{\Rc i}\times\mathcal E_{\Pc i}\times\mathcal E_{\Pc{\rm link}},%\\
%\mathcal F_2=&\, \mathcal F_{\mathcal{L}2}\times \mathcal F_{\Rc2}\times \mathcal F_{\Pc2}\times\mathcal F_{\Pc{\rm link}},\quad \mathcal %E_2=\mathcal F_{\mathcal{L}2}\times\mathcal E_{\Rc2}\times\mathcal E_{\Pc2}\times\mathcal E_{\Pc{\rm link}},
\end{aligned}\]
which are
subdivided into an energy-storing, a~resistive, a~specific external part, and a~to-be-linked part.
We define the \emph{interconnection} of $(\mathcal D_1,\mathcal L_1,\mathcal R_1)$ and $(\mathcal D_2,\mathcal L_2,\mathcal R_2)$,
\[(\mathcal D_1,\mathcal L_1,\mathcal R_1)\circ(\mathcal D_2,\mathcal L_2,\mathcal R_2):=(\mathcal D,\mathcal L,\mathcal R),\]
 \emph{with respect to} $(\mathcal F_{\Pc{\rm link}},\mathcal E_{\Pc{\rm link}})$ as the pH system given by %with
 \[
\begin{aligned}
&\mathcal D\coloneqq \big\{((f_{\Lc1},f_{\Lc2}),(f_{\Rc1},f_{\Rc2}),(f_{\Pc1},f_{\Pc2}),
(e_{\Lc1},e_{\Lc2}),(e_{\Rc1},e_{\Rc2}),(e_{\Pc1},e_{\Pc2}))\\
&\qquad\qquad\qquad~\vert~\exists(f_{{\rm link}},e_{{\rm link}})\in\mathcal F_{{\rm link}}\times\mathcal E_{{\rm link}}:(f_{\Lc1},f_{\Rc1},f_{\Pc1},f_{{\rm link}},e_{\Lc1},e_{\Rc1},e_{\Pc1},e_{{\rm link}})\in\mathcal D_1\\
&\qquad\qquad\qquad\qquad\,\wedge\, (f_{\Lc2},f_{\Rc2},f_{\Pc2},-f_{{\rm link}},e_{\Lc2},e_{\Rc2},e_{\Pc2},e_{{\rm link}})\in\mathcal D_2\big\},
\end{aligned}
\]
and
 \[
\begin{aligned}
\Lc=&\,\{\left.((f_{\Lc1},f_{\Lc2}),(e_{\Lc1},e_{\Lc2}))\in (\mathcal{F}_{\Lc1}\times \mathcal{F}_{\Lc2})\times (\mathcal{E}_{\Lc1}\times \mathcal{E}_{\Lc2})\,\right|\;(f_{\Lc1},e_{\Lc1})\in\mathcal{L}_1\,\wedge\,(f_{\Lc2},e_{\Lc2})\in\mathcal{L}_2\},\\
\Rc=&\,\{\left.((f_{\Rc1},f_{\Rc2}),(e_{\Rc1},e_{\Rc2}))\in (\mathcal{F}_{\Rc1}\times \mathcal{F}_{\Rc2})\times (\mathcal{E}_{\Rc1}\times \mathcal{E}_{\Rc2})\,\right|\;(f_{\Rc1},e_{\Rc1})\in\mathcal{R}_1\,\wedge\,(f_{\Rc2},e_{\Rc2})\in\mathcal{R}_2\}.
\end{aligned}
\]
%
% $\mathcal L\coloneqq\mathcal L_1\times_{\rm r}\mathcal L_2$, $\mathcal R\coloneqq \mathcal R_1\times_{\rm r} \mathcal R_2$ and
%  $\mathcal D\coloneqq(\mathcal D_1\circ\mathcal D_2)_{\rm r}$, where the latter corresponds to a~reordering of the flows and efforts of
%$\mathcal D_1\circ\mathcal D_2$ into energy-storing, resistive and~specific external parts.
\end{definition}
The above constructed set $\mathcal{D}$ is indeed a~Dirac structure \cite[Chap.~6]{JvdS14}. It is obvious that $\mathcal{L}$ is a~Lagrange submanifold and $\mathcal{R}$ is a~resistive relation. Hence, the interconnection of pH systems results in a pH system.\\
%We have seen in Definition~\ref{def-pH} that the dynamics of $(\mathcal D_1,\mathcal L_1,\mathcal R_1)$ and $(\mathcal D_2,\mathcal L_2,\mathcal R_2)$ respectively read
%\[
%\begin{aligned}
%&(-\ddt x_1,f_{\Rc1},f_{\Pc1},f_{\Pc{\rm link}},e_{\Lc1},e_{\Rc1},e_{\Pc1},e_{\Pc{\rm link}})\in\mathcal D_1,\quad
%(x_1,e_{\Lc1})\in\mathcal L_1,\quad (f_{\Rc1},e_{\Rc1})\in\mathcal R_1,\\
%&(-\ddt x_2,f_{\Rc2},f_{\Pc2},f_{\Pc{\rm link}},e_{\Lc2},e_{\Rc2},e_{\Pc2},e_{\Pc{\rm link}})\in\mathcal D_2,\quad
%(x_2,e_{\Lc2})\in\mathcal L_2,\quad (f_{\Rc2},e_{\Rc2})\in\mathcal R_2.
%\end{aligned}
%\]
%Then the dynamics of $(\mathcal D,\mathcal L,\mathcal R)=(\mathcal D_1,\mathcal L_1,\mathcal R_1)\circ(\mathcal D_2,\mathcal L_2,\mathcal R_2)$
%are given by
%\[
%\begin{aligned}
%(-\ddt x(t),f_{\Rc}(t),f_{\Pc}(t),e_{\Lc}(t),e_{\Rc}(t),e_{\Pc}(t))\in\mathcal D,\
%(x(t),e_{\Lc}(t))\in\mathcal L,\ (f_{\Rc}(t),e_{\Rc}(t))\in\mathcal R,
%\end{aligned}
%\]
%where $x=\left(\begin{smallmatrix}x_1\\x_2\end{smallmatrix}\right)$, $f_\Rc=\left(\begin{smallmatrix}f_{\Rc1}\\f_{\Rc2}\end{smallmatrix}\right)$, and according definitions for $f_\Pc$, $e_\Lc$, $e_\Rc$ and $e_\Pc$.\\
  Next we introduce the Cartesian product of pH systems, which simply means that several coexisting pH systems are united to one pH system. In terms of Definition~\ref{def:pHint}, it means that several pH systems are interconnected with trivial linking ports.
  That is, for pH systems
 $(\mathcal D_1,\mathcal L_1,\mathcal R_1)$ and $(\mathcal D_2,\mathcal L_2,\mathcal R_2)$ we add
 artificial and trivial linking ports $\Fc_{\rm link}=\Ec_{\rm link}=\{0\}$ (which do not affect the dynamic behavior) and interconnect these systems with respect to this trivial port $(\mathcal F_{\Pc{\rm link}},\mathcal E_{\Pc{\rm link}})$. A~coupling of this kind will be denoted by
 $(\mathcal D_1,\mathcal L_1,\mathcal R_1)\times (\mathcal D_2,\mathcal L_2,\mathcal R_2)$.
We further inductively define
\[\bigtimes_{i=1}^n(\mathcal D_i,\mathcal L_i,\mathcal R_i)\coloneqq\left(\bigtimes_{i=1}^{n-1}(\mathcal D_i,\mathcal L_i,\mathcal R_i)\right)\times (\mathcal D_n,\mathcal L_n,\mathcal R_n).\]
%This means that for $i=1,ldots,n$,
%\[
%\begin{aligned}
%\forall t\in\R:\ (-\ddt x_i(t),f_{\Rc i}(t),f_{\Pc i}(t),e_{\Lc i}(t),e_{\Rc i}(t),e_{\Pc i}(t))\in\mathcal D_i,\ (x_i(t),e_{\Lc %i}(t))\in\mathcal L_i,\ (f_{\Rc i}(t),e_{\Rc i}(t))\in\mathcal R_i.
%\end{aligned}
%\]
%the dynamics of the product read
%\[
%\begin{aligned}
%\forall t\in\R:\ (-\ddt x(t),f_R(t),f_P(t),e_S(t),e_R(t),e_P(t))\in\mathcal D,\ (x(t),e_S(t))\in\mathcal L,\ (f_R(t),e_R(t))\in\mathcal R,
%\end{aligned}
%\]
%using the notations
%\[x=\begin{pmatrix}x_1\\\vdots\\x_n\end{pmatrix},\quad f_R=\begin{pmatrix}f_{R_1}\\\vdots\\f_{R_n}\end{pmatrix},\quad %f_P=\begin{pmatrix}f_{P_1}\\\vdots\\f_{P_n}\end{pmatrix},\quad e_S=\begin{pmatrix}e_{S_1}\\\vdots\\e_{S_n}\end{pmatrix},\quad %e_R=\begin{pmatrix}e_{R_1}\\\vdots\\e_{R_n}\end{pmatrix},\quad e_P=\begin{pmatrix}e_{P_1}\\\vdots\\e_{P_n}\end{pmatrix}.\]%

\begin{figure}
	% \centering
	\begin{minipage}[t]{0.45\textwidth}
		\centering
		\includegraphics[width=\textwidth]{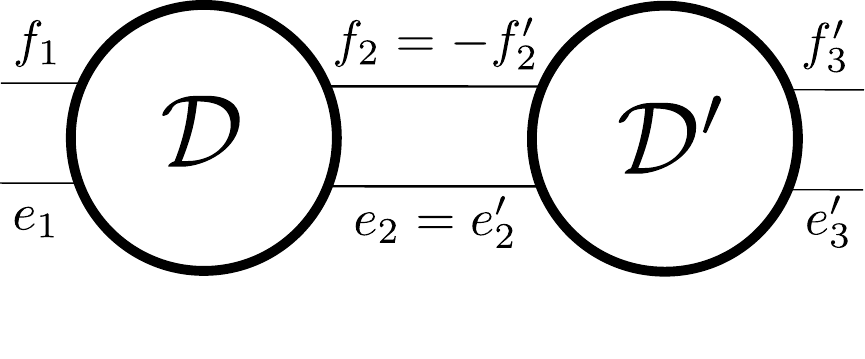}
		\caption{Composition of two Dirac structures.}
		% \label{composition}
	\end{minipage}\hfill
	\begin{minipage}[t]{0.45\textwidth}
		\centering
		\includegraphics[width=\textwidth]{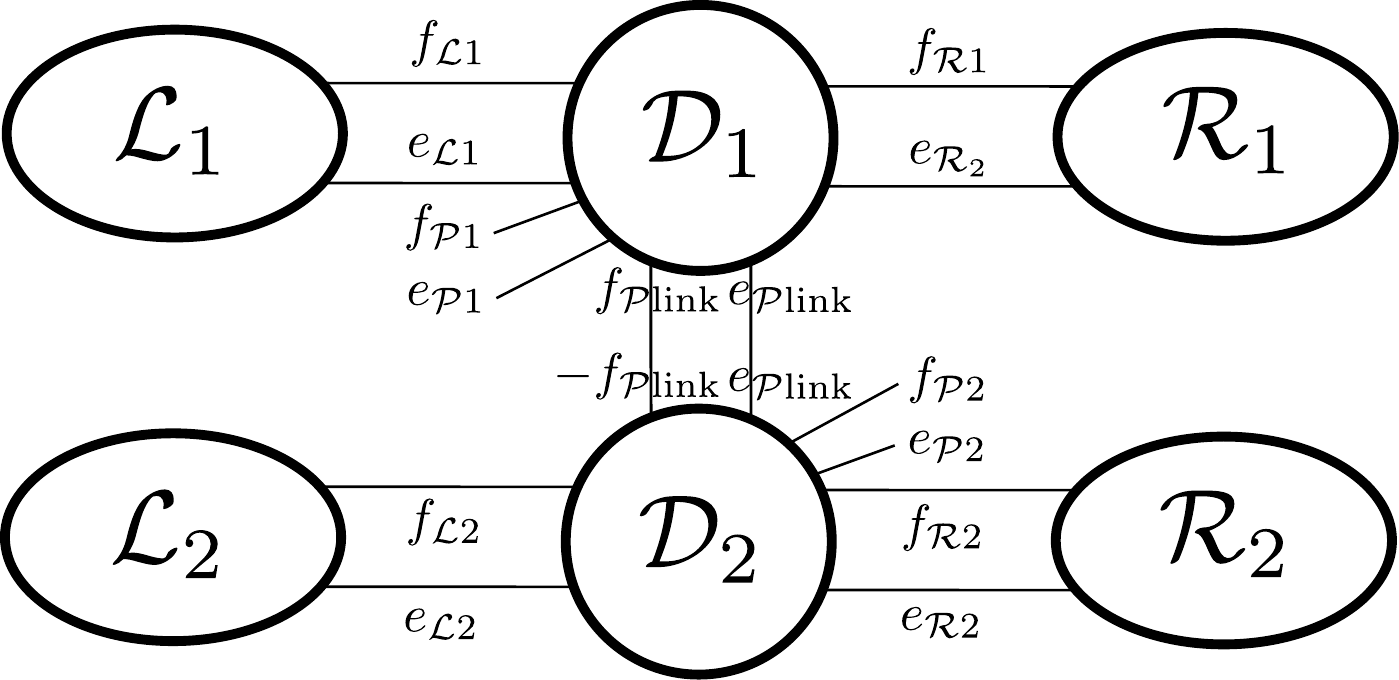}
		\caption{Interconnection of two pH systems.}
		% \label{interconnection}
	\end{minipage}
\end{figure}

%\feh{
%	\begin{remark}
%		It readily follows from Definition \ref{def:pHint} that the Lagrange submanifold of the interconnection can be characterized as the cartesian product of the Lagrange submanifolds of the pH-systems which have been interconnected. A characterization of the resulting Dirac structure in not that evident. However, in the modelling of electrical circuits all Dirac structures involved in interconnections will be given in kernel representation \eqref{eq:kernel}. For such interconnections, \cite[Theorem 4]{CvdSB07} gives a method to compute a kernel representation \eqref{eq:kernel} of the Dirac structure of the interconnection.\\\ \\
%	\end{remark}
%}

\subsection{Port-Hamiltonian systems on graphs}
 \noindent Now we consider interconnections of pH systems, which are defined via graphs \cite{vdSM13a}.
  This will lead us to the notions of \emph{Kirchhoff-Dirac structure} and \emph{Kirchhoff-Lagrange manifold}. Later we will show that such interconnections correspond to
 the Kirchhoff laws in electrical circuits.
%
% port-Hamiltonian systems which are interconnected via graphs as considered in . Different Dirac structures reflecting this topology can be derived. For the synthesis of electrical circuits, one such Dirac structure is , with which we associate the . Together, they define a port-Hamiltonian system on the circuit graph which reflects Kirchhoff's current law.
%In contrast to \cite{vdSM13a}, where the Kirchhoff-Dirac structure is directly used to model the port-Hamiltonian system, we use the Kirchhoff-Dirac structure more as an interconnection structure in the sense of \cite{}.
To this end, we introduce some basic graph theoretical notions from \cite{Dies17}.
\begin{definition}[Graphs and subgraphs]\!\!
A \emph{directed graph} is a quadruple $\mathcal{G} = (V,E,\init,\ter)$ consisting of a \emph{vertex set} $V$, a \emph{edge set} $E$ and two maps $\init, \ter: E \rightarrow V$ as\-signing to each edge $e$ an {\em initial vertex} $\init(e)$ and a {\em terminal vertex} $\ter(e)$. The edge $e$ is said to be {\em directed from $\init(e)$ to $\ter(e)$}. $\mathcal{G}$ is said to be {\em loop-free}, if $\init(e)\neq\ter(e)$  for all $e\in E$. Let $V' \subset V$ and $E' \subset E$ with
\[ E' \subset \left. E\right|_{V'} := \{ e \in E : \init(e)\in V'\,\wedge\, \ter(e) \in V' \}. \]
Then the triple $(V',E',\left. \init\right|_{E'},\left. \ter\right|_{E'})$ is called a \emph{subgraph of $\mathcal{G}$}. If $E' = \left. E\right|_V'$, then the subgraph is called the \emph{induced subgraph} on $V'$.
If $V' = V$, then the subgraph is called \emph{spanning}.
Additionally a \emph{proper subgraph} is one where $E' \neq E$.
$\mathcal{G}$ is called {\em finite}, if $V$ and $E$ are finite.
\end{definition}
The notion of a~{\em path} in a~directed graph $\mathcal{G} = (V,E,\init,\ter)$ is quite descriptive. However, since a~path may also go through an edge in reverse direction, we define for each $e \in E$ an additional edge $-e \not\in E$ with $\init(-e)=\ter(e)$ and $\ter(-e)=\init(e)$.
\begin{definition}[Paths, connectivity, cycles, forests and trees]\!\!
Let $\mathcal{G} = (V,E,\init,\ter)$ be a~directed finite graph. An $r$-tuple $e = (e_1,\ldots,e_r) \in ({E}\cup -E)^r$ is called a \emph{path from $v$ to $w$}, if% for all $i,j\in\{1,\ldots,r\}$ with $i\neq j$ holds $e_i\neq e_j$ and $e_i\neq -e_j$, and
\[\begin{aligned}
&\init(e_1),\ldots,\init(e_r)\text{ are distinct},\\
&\ter(e_i) = \init(e_{i+1}) \; \;  \forall i \in \{ 1,\ldots, r-1 \},\\
&\init(e_1)=v\,\wedge\,\ter(e_r)=w.
\end{aligned}\]
%If the vertices are all distinct from one another, it is called an \emph{elementary path}.
A \emph{cycle} is a~path from $v$ to $v$. Two vertices $v,w$ are \emph{connected}, if there is a path from $v$ to $w$. This gives is an equivalence relation on the vertex set. The induced subgraph on an equivalence class of connected vertices gives a \emph{component} of the graph. A graph is called {\em connected}, if there is only one component.\\
A subgraph $\mathcal{K} = (V,E', \left. \init\right|_{E'}, \left. \ter\right|_{E'})$ of a~directed graph $\mathcal{G}=(V,E,\init,\ter)$ is called a \emph{spanning forest} in $\mathcal{G}$, if $\mathcal{K}$ does not contain any cycles and is maximal with this property, that is, $\mathcal{K}$ is not a~proper subgraph of a~subgraph of $\mathcal{G}$ which does not contain any cycles.
A~subgraph $\mathcal{K}$ is called {\em tree}, if it is a~forest and connected.
%Consider a directed graph $\mathcal{G}$ with spanning subgraph $\mathcal{K}$. We call a~subgraph $\mathcal{L}$ of $\mathcal{G}$ a \emph{$\mathcal{K}$-cut}, if $\mathcal{L}$ is a cut of $\mathcal{K}$. Further, we call a~path in $\mathcal{G}$ a {\em $\mathcal{K}$-cycle}, if it is a cycle in $\mathcal{K}$.\\
%If  $\mathcal{K}_1$ and $\mathcal{K}_2$ are two spanning subgraphs $\mathcal{G}$, then $\mathcal{K}_1\mathcal{K}_2$ denotes the spanning subgraph obtained by taking the union of the edges $\mathcal{K}_1$ and $\mathcal{K}_2$.
\end{definition}
In the context of electrical circuits, finite and loop-free directed graphs are of major importance. These allow to associate a~special matrix \cite[Sec.~3.2]{And91}.
\begin{definition}[Incidence matrix]
Let $\mathcal{G}=(V,E,\init,\ter)$ be a finite and loop-free directed graph. Let $E=\{e_1,\ldots,e_m\}$ and $V=\{v_1,\ldots,v_n\}$. Then the \emph{incidence matrix} of $\mathcal{G}$ is $A_0\in\R^{n\times m}$ with
\[a_{jk}=\begin{cases}1&\init(e_k)=v_j, \\-1&\ter(e_k)=v_j,\\0&\text{otherwise.}\end{cases}\]
\end{definition}
$\mathcal{G}$ has $k\in\N$ components if, and only if, $\rk A_0=n-k$ \cite[p.\ 140]{And91}. This allows to remove up to $k$ rows from $A_0$ such that a~matrix with same rank is obtained. The choice of these to-be-deleted rows has to be done in a~special way: One has to choose a~row set, which corresponds to a vertex set $S$ that contains at most one vertex per component to $\mathcal{G}$. This deletion plays a~crucial role in the following definition of a~special Dirac structure and Lagrange submanifold.

\begin{definition}[Kirchhoff-Dirac structure, Kirchhoff-Lagrange submanifold]\label{def-K}
Assume that $\mathcal{G}=(V,E,\init,\ter)$ is a~finite and loop-free directed graph with
 incidence matrix $A_0\in\R^{n\times m}$. Let $\mathcal{G}_1,\ldots,\mathcal{G}_k$ be the components of $\mathcal{G}$ and let $V_1,\ldots,V_k\subset V$ be the corresponding vertex sets. Let $S\subset V$
 such that $S$ contains at most one vertex form each component, that is
\begin{equation}\label{eq:S}
\forall\, s,s'\in S,\, i\leq k:\;\; v,v'\in V_i\Rightarrow v=v'.
\end{equation}
Let $A\in\R^{(n-k)\times m}$ be constructed from $A_0\in\R^{n\times m}$ by deleting the rows corresponding to the vertices from $S$.
The \emph{Kirchhoff-Dirac structure} of $\mathcal{G}$ is
\begin{equation}\label{KDS}
\begin{aligned}
\mathcal D^S_K(\mathcal{G})\coloneqq\bigg\{(j,i,\phi, u)\in\R^{n-|S|}\times\R^m\times\R^{n-|S|}\times\R^m~\bigg\vert~\begin{bmatrix}I&A\\0&0\end{bmatrix}\begin{pmatrix}j\\ i\end{pmatrix}+\begin{bmatrix}0&0\\A^\top&-I\end{bmatrix}\begin{pmatrix}\phi\\u\end{pmatrix}=0\bigg\}.
\end{aligned}
\end{equation}
Assume that $S=\{v_1,...,v_{|S|}\}$ (which is - by a reordering of the vertices - no loss of generality). Then the \emph{Kirchhoff-Lagrange submanifold} of $\mathcal{G}$ with respect to $S$ is
\begin{equation}\label{KLM}
\mathcal L^S_K(\mathcal{G})\coloneqq \{0\}\times\R^{n-|S|}\subset \R^{n-|S|}\times\R^{n-|S|}.
\end{equation}
\end{definition}

\begin{remark}\label{rem-inv}
By Proposition \ref{prop-kernel}, $\mathcal D^S_K(\mathcal{G})$ in \eqref{KDS} is a~Dirac structure, whereas Proposition~\ref{prop-kernel_L} implies that $\mathcal L^S_K(\mathcal{G})$ in \eqref{KLM} is a~Lagrange submanifold of $\R^{n-|S|}\times\R^{n-|S|}$.\\
The concepts of Definition \ref{def-K} allow to introduce the pH system $(\mathcal D^S_K(\mathcal{G}),\mathcal L^S_K(\mathcal{G}),\{0\})$
with dynamics
\begin{equation}\label{eq:KirchpH}
 (-\ddt q(t),i(t),\phi(t),u(t))\in\mathcal D^S_K(\mathcal{G}),\quad (q(t),\phi(t))\in\mathcal L^S_K(\mathcal{G}).
\end{equation}
%
%. To see that (\ref{KLM}) indeed defines a Lagrange submanifold, first note that the identifications (\ref{ident}) enables us to regard
%\[{\mathcal L}^{-1}=\{(w,v)\in\R^n\times\R^n~|~(v,w)\in\mathcal L\}\]
%as a Lagrange submanifold of $T^*\R^n$ if (and only if) $\mathcal L$ is. Further, the product of two Lagrange submanifold
%\[\mathcal L_1\times\mathcal L_2\] can also be regarded as a Lagrange submanifold. Now, well-definition of (\ref{KLM}) follows from the observation that we can write
%\[\mathcal L^S_K(G)={\mathcal L}^{-1}\times\mathcal L',\]
%where $\mathcal L,\mathcal L'$ are both graphs of gradients of trivial functions (cf. Proposition \ref{prop-gradient}). If $S$ is not ordered as %$\{v_1,...,v_{|S|}\}$, we would define (\ref{KLM}) as
%\begin{equation*}
%\mathcal L^S_K(G)\coloneqq\left\{(f_\phi,e_\phi)\in\R^n\times\R^n~|\forall i\in\{1,\ldots,n\}: (v_i\in V\setminus S\Rightarrow f_{\phi_i}=0)\wedge %(v_i\in S\Rightarrow e_{\phi_i}=0)\right\},
%\end{equation*}
%and a similar argument would show, that this still defines a Lagrange submanifold.
%\[
%\forall t\in\R:\ (-\ddt f_\phi(t),\tilde i(t),\phi(t),\tilde u(t))\in\mathcal D_K(\mathcal{G}),\quad (f_\phi(t),\phi(t))\in\mathcal %L^S_K(\mathcal{G}),
%\]
%\[(-\ddt f_\phi(t),\tilde i(t),\phi(t),\tilde u(t))\in\mathcal D_K(\mathcal{G})\]
%if and only if
%\[\ddt f_\phi(t)=A_0\tilde i(t)\;\wedge\;A_0^\top \phi(t) =-\tilde u(t)\]
Then, by the equivalence of $(q(t),\phi(t))\in\mathcal L^S_K(\mathcal{G})$ to $q(t)=0$ and $\phi(t)\in\R^{n-|S|}$, we see that \eqref{eq:KirchpH} holds, if, and only if,
\[q(t)=0\,\wedge\, A i(t)=0\,\wedge\, A^\top \phi_2(t)=-u(t).\]
In particular, $i(t)\in \ker A$ and $u(t)\in \im A^\top$. In the context of electrical circuits, this will indeed represent Kirchhoff's current and voltage law \cite[Thm.~4.5 \& Thm.~4.6]{Rei14}. The choice of $S$ can be interpreted as the set of \emph{grounded vertices}. The quantities $q$, $i$, $\phi$ and $u$ can respectively be thought as the vertex charges, the edge currents, the vertex potentials, and the edge voltages.

Note that \eqref{eq:KirchpH} is indeed a pH system. However, this system is of rather pathological nature, since it does not contain any `true dynamics', as the differential variable $q$ is nulled by the Lagrange submanifold.
Note that these `true dynamics' come into play later on, when we interconnect with dynamic circuit elements like capacitances and inductances.

\noindent
In the terminology of \cite{vdSM13a}, $\mathcal D_K^S(\mathcal{G})$ corresponds to the Kirchhoff-Dirac structure of a graph with $|S|=\emptyset$. Moreover, a~Dirac structure similar to \eqref{KDS} has been used in \cite{vdS10}, with the main difference that in our present case all nodes are considered to be `boundary nodes' in the nomenclature of \cite{vdS10}.
\end{remark}
%from which one can derive $A_0\tilde i(t)=0$.
%If $S$ contains more than one vertex per component, the $\mathcal L_K^S(G)$ would still define a Lagrange submanifold, but we would be unable to derive Kirchhoff's current law.
We briefly present an alternative (slightly less straight-forward) construction of pH systems on graphs, namely by means of cycles instead to vertices.
%We omit certain the graph-theoretical definitions but indicate where to find them in the literature.
For a given spanning forest $\mathcal{T}$ of a~loop-free directed graph $\mathcal{G}$ with $n$ edges, $m$ vertices and $k$ connected components, the minimality property yields that the incorporation of any edge of $\mathcal{G}$ not belonging to $\mathcal{T}$ (called {\em chord}) results in a~subgraph with exactly one cycle. Consequently, the set of edges in the complement of $\mathcal{T}$ in $\mathcal{G}$ leads to a~set $\mathcal C=\{\mathcal{C}_1,\ldots,\mathcal{C}_{m-n+k}\}$ of cycles, the so-called  {\em fundamental cycles} (see \cite[p.\ 148]{And91} \& \cite[p.\ 26]{Dies17}). We equip each fundamental cycle with the orientation of its corresponding chord \cite[p.\ 148]{And91} and consider the associated \emph{fundamental cycle matrix} $B\in\R^{(m-n+k)\times m}$ which is defined entrywise by (cf. \cite[Sec.~3.3]{And91})
\[b_{jl}=\begin{cases}1&e_l\in C_j \text{ and the orientations agree,} \\-1&e_l\in C_j \text{ and the orientations do not agree,} \\0&\text{otherwise.}\end{cases}\]
This enables us to introduce the following Dirac structure and Lagrange submanifold
\begin{equation}\label{KDS'}
\begin{aligned}
\mathcal D'_K(\mathcal{G})\coloneqq&\,\bigg\{(\varphi,u,\iota, i)\in\R^{m-n+k}\times\R^m\times\R^{m-n+k}\times\R^m~\bigg\vert~\begin{bmatrix}I&B\\0&0\end{bmatrix}\begin{pmatrix}\varphi\\ u\end{pmatrix}+\begin{bmatrix}0&0\\B^\top&-I\end{bmatrix}\begin{pmatrix}\iota\\i\end{pmatrix}=0\bigg\},\\
\mathcal L'_K(\mathcal{G})\coloneqq&\,\{0\}\times\R^{n-m+k},
\end{aligned}
\end{equation}
which form the pH system $(\mathcal D'_K(\mathcal{G}),\mathcal L'_K(\mathcal{G}),\{0\})$
with dynamics
\begin{equation}\label{eq:loop_pH}
(-\ddt \psi(t), i(t),\iota(t), u(t))\in\mathcal D'_K(\mathcal{G}),\quad (\psi(t),\iota(t))\in\mathcal L'_K(\mathcal{G}),
\end{equation}
from which, analogous to Remark~\ref{rem-inv}, one can derive that \eqref{eq:loop_pH} is equivalent to $\psi(t)=0$, $Bu(t)=0$ and $i(t)=B^\top \iota$. Since $\im B=\ker A^\top$ \cite[Thm.~4.4]{Rei14}, the relations $u(t)\in\ker B=0$ and $i(t)\in\im B^\top$ respectively represent Kirchhoff's voltage and current law. The quantities $\psi$, $u$, $\iota$ and $i$ can respectively be thought as the cycle fluxes, the edge voltages, the cycle currents and the edge currents.

\section{Electrical circuits as port-Hamiltonian systems}
\noindent
Our essential idea to port-Hamiltonian modelling of electrical circuits is to extend the tuple of voltages across and currents through the edges - in the case where we consider a~vertex-based formulation of the Kirchhoff laws - by vertex charges and potentials, and - in the case where we consider a~loop-based formulation of the Kirchhoff laws - by cycle fluxes and cycle currents, along with an accordant modelling of the graph interconnection structure by means of the approach in the preceding section. The electrical components are modelled by separate pH systems, and thereafter coupled with the one representing the interconnection structure.\\
The circuits may be composed of \emph{two-terminal} and \emph{multi-terminal} components.
We will speak of \emph{$\ell_t$-terminal components}, with $\ell_t\in\N$ denoting the number of terminals \cite{Wil10}.
Each $\ell_t$-terminal component connects $\ell_t$ vertices of the electrical circuit through its terminals. For instance, a resistance has two terminals, whereas a~transistor has three terminals, and a~transformer has four terminals. To regard an electrical circuit as a graph (see Fig.~\ref{circuit-to-graph}), we need to replace the $\ell_t$-terminal components by $\ell_p$ edges between the vertices they connect, for some $\ell_p\in\N$, which we call the number of ports. Such a~device is also called a~{\em $\ell_p$-port component}. This replacement is displayed in Fig.~\ref{circuit-to-graph}. The direction assigned to each edge is not a physical restriction but rather a definition of the \emph{positive direction} of the corresponding voltage and current \cite{Rei14}. The physical properties of the electrical components will be reflected by port-Hamiltonian dynamics on these edges. The replacement of an $\ell_t$-terminal component by $\ell_p$ edges between vertices, i.e., by a graph, is subject to physical modelling. For further details on terminals, ports and their relation, we refer to \cite{Wil10}.

\begin{figure}
	% \centering
	\begin{minipage}[t]{0.45\textwidth}
		\centering
		\includegraphics[scale=0.57]{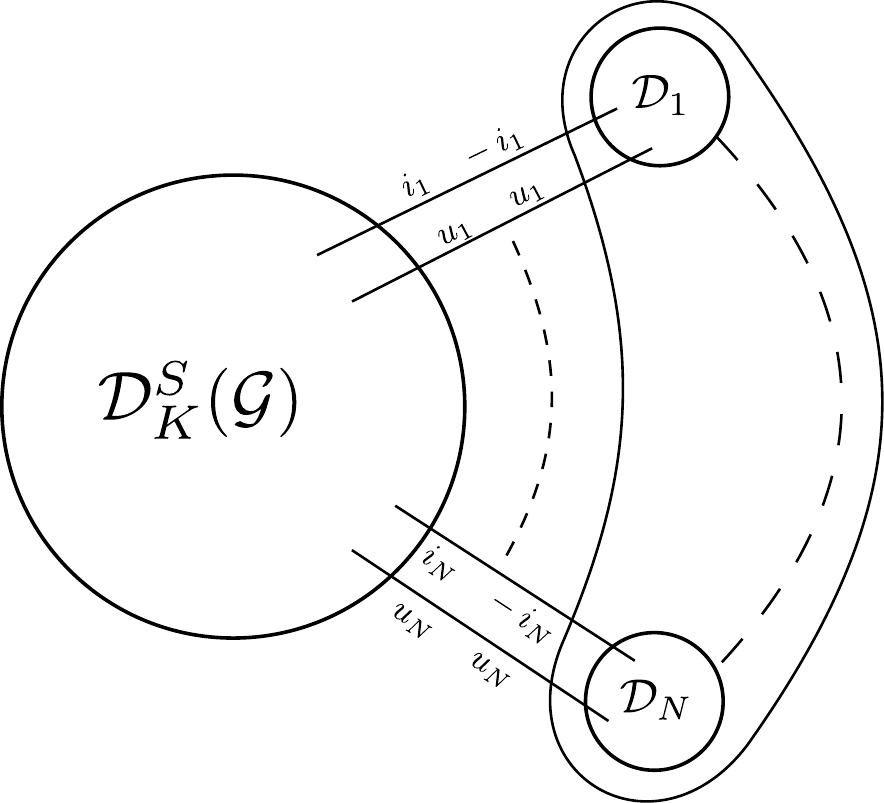}
		\caption{Visual representation of the Dirac structure $\mathcal D$ resulting from the interconnection (\ref{circ-model}).}
		% \label{Circuit D}
	\end{minipage}\hfill
		\centering
    	\begin{minipage}[t]{0.45\textwidth}
		\quad\;\includegraphics[scale=0.35]{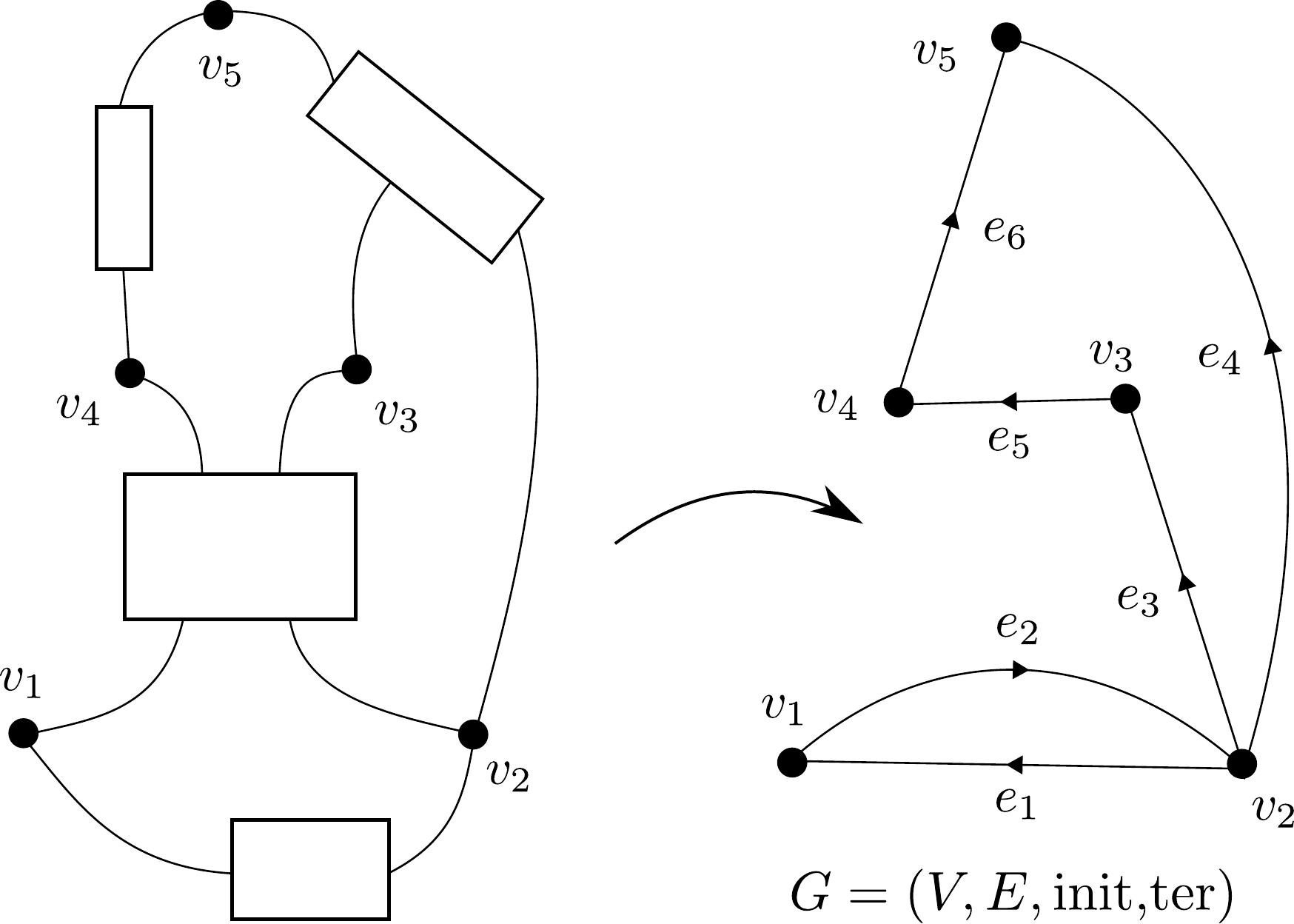}
		\caption{Obtaining the underlying graph of the electrical circuit.}
		\label{circuit-to-graph}
	\end{minipage}
\end{figure}

To be more precise, for $\ell_p,\ell_t\in\N$, an \emph{$\ell_t$-terminal component on $\ell_p$ edges} will be regarded as a pH system $(\mathcal D,\mathcal L,\mathcal R)$, where
$\mathcal D\subset \R^{n_S+n_R+\ell_p}\times\R^{n_S+n_R+\ell_p}$,
with
$\ell_p=n_S+n_R$ for some $n_S,n_R\in\N_0$. We associate to $\mathcal D$ a graph $\mathcal{G}=(V,E,\text{init},\text{ter})$ with $|V|=\ell_t$ and $|E|=\ell_p$ (cf. Fig.~\ref{comp-to-graph}). The external flow and effort variables will always represent the current through \cite[Def.~3.2]{Rei14} and the voltage along \cite[Def.~3.6]{Rei14} the corresponding edges, respectively.

\subsection{Electrical circuits as interconnections of port-Hamiltonian systems}\label{subsec-comp}
\noindent Let an electrical circuit consisting of $N$ electrical components $(\mathcal D_i,\mathcal L_i,\mathcal R_i)_{i\in\{1,\ldots, N\}}$, each with $\ell_{p,i}$ ports, be given, with $N\in\N$ and let $(\mathcal{G}_i)_{i\in\{1,\ldots, N\}}=(V_i,E_i,\text{init}_i,\text{ter}_i)_{i\in\{1,\ldots, N\}}$ be the respective graphs resulting from the physical modelling of the $\ell_{p,i}$-port components (see Fig.~\ref{comp-to-graph}), where we assume that the edge sets $E_1,\ldots,E_N$ are disjoint. We define the \emph{underlying graph of the circuit} $\mathcal{G}$ (see Fig.~\ref{circuit-to-graph}) as
\[\mathcal{G}=(V,E,\text{init},\text{ter})\coloneqq\left(\bigcup_{i=1}^N V_i,\bigcup_{i=1}^N E_i,\text{init},\text{ter}\right),\]
with $\text{init}(e)=\text{init}_i(e)$ and $\text{ter}(e)=\text{ter}_i(e)$ if $e\in E_i$ for some $i\in\{1,...,N\}$ and let $V=\{v_1,\ldots,v_n\}$, $E=\{e_1,\ldots,e_m\}$ for some $n,m\in\N$. Further, let $A_0\in\R^{n\times m}$ be the incidence matrix associated to $\mathcal{G}$ and let $S\subset V$ with property \eqref{eq:S} represent the vertices grounded in the circuit. We model the dynamics of the electrical circuits as the dynamics of the pH system
\begin{equation}\label{circ-model}
(\mathcal D,\mathcal L,\mathcal R)\coloneqq(\mathcal D^S_K(\mathcal{G}),\mathcal L^S_K(\mathcal{G}),\{0\})\circ\left(\bigtimes_{i=1}^N(\mathcal D_i,\mathcal L_i,\mathcal R_i)\right),
\end{equation}
where the interconnection is performed with respect to the flow and effort spaces \[(\mathcal F_{\rm link},\mathcal E_{\rm link})=\left(\bigtimes_{i=1}^N \R^{m_{P_i}},\bigtimes_{i=1}^N \R^{m_{P_i}}\right)=(\R^m,\R^m)\] corresponding to the port variables associated to the currents and voltages of the $\ell_p$-port components.

\subsection{Physical modelling of circuit components as port-Hamiltonian systems}
\noindent
We present a~couple of `prominent' electrical components from a port-Hamiltonian viewpoint; among them are capacitances, inductances, resistances, diodes, transformers, transistors and sources. Note that this list
is by no means complete. In principle, our approach also allows to incorporate components which are modelled by partial differential equations, such as transmission lines and refined models of semiconductor devices. This involves a~further generalization of pH systems on infinite-dimensional spaces and particularly leads to the notion of {\em Stokes-Dirac structure}, see \cite{BKvdSZ10,MvdS04a,MvdS04b}.\\
Throughout this section, $i$ will denote currents and $u$ will denote voltages.
An oftentimes used Dirac structure will be, for $\ell_p\in\N$,
%\begin{equation}
%\mathcal D_{k} =\left\{\begin{pmatrix}i\\j\\u\\\phi\end{pmatrix}\in\R^{4k}~\bigg\vert~\begin{bmatrix}I_k&I_k\\0&0\end{bmat%rix}\begin{pmatrix}i\\j\end{pmatrix}+\begin{bmatrix}0&0\\I_k&%-I_k\end{bmatrix}\begin{pmatrix}u\\\phi\end{pmatrix}=0\right\%}
%\label{eq:Diracn}
%\end{equation}
\begin{equation}
\mathcal D_{\ell_p} =\left\{\left(\begin{smallmatrix}-i\\i\\u\\u\end{smallmatrix}\right)\in\R^{4\ell_p}~\bigg\vert\; i,u\in\R^{\ell_p}\right\}.
\label{eq:Diracn}
\end{equation}
It can easily verified that this is indeed a Dirac structure.
The variable $i$ stands for the vector of currents, whereas $u$ is the vector of voltages in the component. Note that a copy of the voltage and negative of the current vector is required, since it is later on eliminated by the interconnection according to Definition~\ref{def:pHint}.
\subsubsection{Capacitances}

\begin{figure}
	\begin{minipage}[t]{0.45\textwidth}
		\centering
		\includegraphics[width=0.9\textwidth]{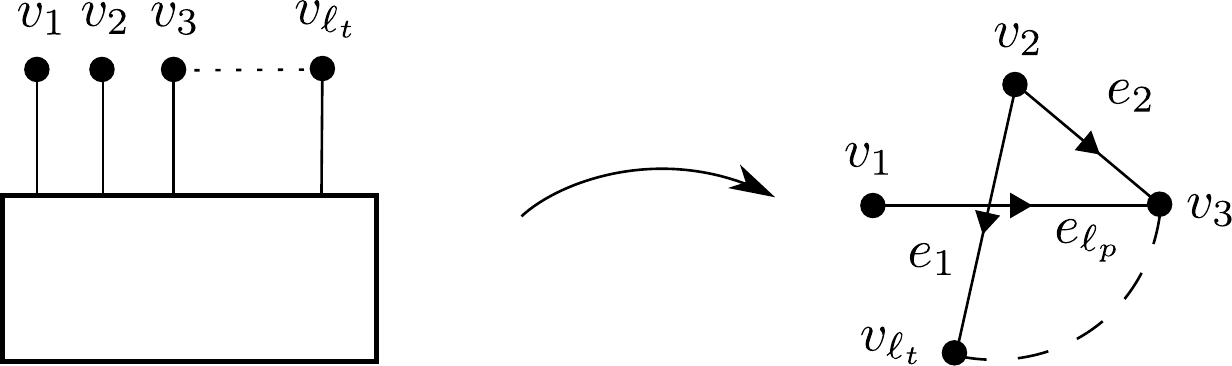}
		\caption{Replacing an $\ell_t$-terminal component by a graph with $\ell_p$ edges.}
		\label{comp-to-graph}
	\end{minipage}\hfill
	\begin{minipage}[t]{0.45\textwidth}
		\centering
		\includegraphics[width=\textwidth]{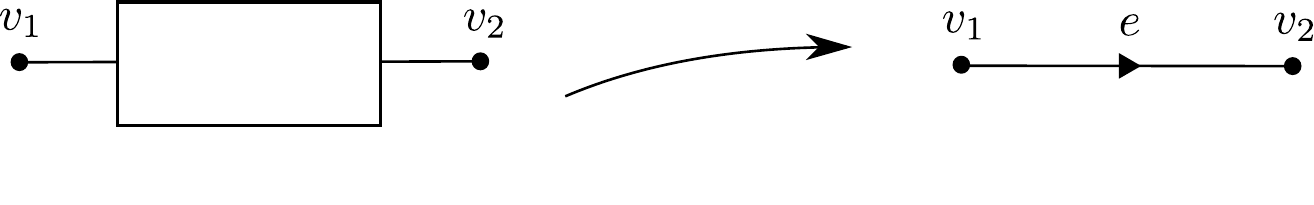}
		\caption{Deriving the underlying graph of a capacitance, conductance, ideal diode, PN-junction diode, inductance, resistance, or sources.}
		\label{2port-to-graph}
	\end{minipage}
\end{figure}

\noindent Let $H_{\cC}\in C^1(\R^{\ell_p},\R)$. A capacitance with $\ell_p$ ports is modelled as a~pH system $(\mathcal D_{\cC},\mathcal L_{\cC},\mathcal R_{\cC})$, where $\mathcal D_{\cC}=\mathcal D_{\ell_p}$ with $\mathcal D_{\ell_p}$ as in \eqref{eq:Diracn}, $\mathcal{R}_{\cC}=\{0\}$, and
\[
\mathcal{L}_{\cC}=\left\{(u_{\cC},q_{\cC})\in\R^{2\ell_p}\vert\ q_{\cC}=\nabla H_{\cC}(u_{\cC})\right\}.
\] The dynamics consequently read
\[(-\ddt q_{\cC}(t),i_{\cC}(t),u_{\cC}(t),u_{\cC}(t))\in\mathcal D_{\cC},\quad (q_{\cC}(t),u_{\cC}(t))\in\mathcal L_{\cC}.\]
Here, $q_{\cC}$ represents the \emph{charge} of the capacitance and the \emph{Hamiltonian} $H_{\cC}$ represents the \emph{energy storage function} of the system. From this pH system, one can derive
\begin{equation*}%\label{Relation-C}
	i_{\cC}(t)=\ddt q_{\cC}(t),\quad u_{\cC}(t)=\nabla H_{\cC}(q_{\cC}(t)).
\end{equation*}
If the capacitance has two terminals, then we obtain a~conventional capacitance with one port as in Fig.~\ref{2port-to-graph}.
\subsubsection{Inductances}

\begin{figure}
	% \centering
	\begin{minipage}[t]{0.33\textwidth}
	\centering
~\\[-1.7cm]		\begin{circuitikz}%[every loop/.style={<-,shorten <=1pt,min distance=12mm}]
		\draw
			(4,0) to[C=$\cC$,i^<=$\quad i_{\cC}$,v_<=$u_{\cC}$,*-*] (0,0);
		\end{circuitikz}~\\[-0.25cm]
		\caption{Capacitance: circuit symbol}
		% \label{C-symbol}
	\end{minipage}\hfill
	\begin{minipage}[t]{0.33\textwidth}
		\centering
		\begin{circuitikz}[european voltages, american currents, european resistors]]%[every loop/.style={<-,shorten <=1pt,min distance=12mm}]
		\draw
			(4,0) to[L=$\cL$,i^<=$\quad i_{\cL}$,v_<=$u_{\cL}$,*-*] (0,0);
		\end{circuitikz}
		\caption{Inductance: circuit symbol}
		% \label{L-symbol}
	\end{minipage}
	% \centering
	\begin{minipage}[t]{0.33\textwidth}
		\centering
		\begin{circuitikz}%[every loop/.style={<-,shorten <=1pt,min distance=12mm}]
		\draw
			(4,0) to[R=$\cR$,i^<=$\quad i_{\cR}$,v_<=$u_{\cR}$,*-*] (0,0);
		\end{circuitikz}
		\caption{Resistance/conductance: circuit symbol}
		 \label{G-symbol}
	\end{minipage}\hfill
%	\begin{minipage}[t]{0.45\textwidth}
%		\centering
%		\begin{circuitikz}%[every loop/.style={<-,shorten <=1pt,min distance=12mm}]
%		\draw%
%			(4,0) to[R=$\cR$,i^<=$\quad i_{\cR}$,v_<=$u_{\cR}$,*-*] (0,0);
%		\end{circuitikz}%
%		\caption{Circuit symbol of a resistance.}
%		% \label{R-symbol}
%	\end{minipage}
\end{figure}

\noindent
Let $H_{\cL}\in C^1(\R^{\ell_p},\R)$.
An inductance with $\ell_p$ ports is modelled as a~pH system $(\mathcal D_{\cL},\mathcal L_{\cL},\mathcal R_{\cL})$
with
\[
\mathcal D_{\cL} =\left\{\left(\begin{smallmatrix}-u_{\cL}\\i_{\cL}\\i_{\cL}\\u_{\cL}\end{smallmatrix}\right)\in\R^{4\ell_p}~\bigg\vert~u_{\cL},i_{\cL}\in\R^{\ell_p}\right\}
\]
and
\[\mathcal{L}_{\cL}=\left\{(\psi_{\cL},i_{\cL})\in\R^{2\ell_p}\vert\ i_{\cL}=\nabla H_{\cL}(\psi_{\cL})\right\},\quad \mathcal R_{\cL}=\{0\}.\]
The dynamics are now given by
\[(-\ddt\psi_{\cL}(t),i_{\cL}(t),i_{\cL}(t),u_{\cL}(t))\in\mathcal D_{\cL},\quad (\psi_{\cL}(t),i_{\cL}(t))\in\mathcal L_{\cL},\]
Here, $\psi_{\cL}$ represents the \emph{magnetic flux} of the inductance and the \emph{Hamiltonian} $H_{\cL}\in C^1(\R^{\ell_p},\R)$ represents the \emph{energy storage function} of the system. From this pH system, one can derive
 \begin{equation*}%\label{Relation-L}
	u_{\cL}(t)=\ddt \psi_{\cL}(t),\quad  i_{\cL}(t)=\nabla H_{\cL}(\psi_{\cL}(t)).
\end{equation*}
If the inductance has two terminals, then we obtain a~conventional inductance with one port as in Fig.~\ref{2port-to-graph}.

\subsubsection{Conductances and resistances}

\noindent Let $\mathcal R_{\cR}\subset\R^{\ell_p}\times\R^{\ell_p}$ be a~resistive relation. Consider the pH system
$(\mathcal D_{\cR},\mathcal L_{\cR},\mathcal R_{\cR})$, where $\mathcal D_{\cR}=\mathcal D_{\ell_p}$ with $\mathcal D_{\ell_p}$ as in \eqref{eq:Diracn}, $\mathcal L_{\cR}=\{0\}$.
The dynamics are specified by
\begin{equation}
(-i_{\cR}(t),i_{\cR}(t),u_{\cR}(t),u_{\cR}(t))\in\mathcal D_{\cR},\quad (-i_{\cR}(t),u_{\cR}(t))\in\mathcal R_{\cR},\label{eq:res}
\end{equation}
If, for some accretive function $g:\R^{\ell_p}\to\R^{\ell_p}$ (that is, $\phi_{\cR}^\top g(\phi_{\cR})\geq0$ for all $\phi_{\cR}\in\R^{\ell_p}$), $\mathcal R_{\cR}$ reads
\[\mathcal R_{\cR} =\left\{(-i_{\cR},u_{\cR})\in\R^{2\ell_p}\vert i_{\cR}=g(u_{\cR})\right\},\]
then \eqref{eq:res} leads to $i_{\cR}(t)=g(u_{\cR}(t))$. That is, $(\mathcal D_{\cR},\mathcal L_{\cR},\mathcal R_{\cR})$ describes a~conductance with $\ell_p$ ports. On the other hand, if for some accretive function $r:\R^{\ell_p}\to\R^{\ell_p}$,
\[\mathcal R_{\cR} =\left\{(-i_{\cR},u_{\cR})\in\R^{2\ell_p}\vert u_{\cR}=r(i_{\cR})\right\}\]
then \eqref{eq:res} leads to $u_{\cR}(t)=r(i_{\cR}(t))$, i.e.\ $(\mathcal D_{\cR},\mathcal L_{\cR},\mathcal R_{\cR})$ models a~resistance with $\ell_p$ ports.\\
If the conductance/resistance has two terminals, then we obtain a~conventional conductance/resistance with one port as in Fig.~\ref{G-symbol}.

\begin{remark}
Resistances form a~pathological case of a~pH system, since the underlying Lagrange submanifold is trivial (cf.\ Remark~\ref{rem-inv}). Therefore, the `dynamics' of the pH system are actually `statics'.
The same holds for the models diodes, transformers and transistors which are discussed in the sequel.
\end{remark}
%
%A conductance is modelled as a $2$-terminal component $(\mathcal D_{\cR},\mathcal L_{\cR},\mathcal R_{\cR})$ on one edge (Fig.~\ref{2port-to-graph}) whose
%for all $t\in\R$, where
%\[
%\mathcal D_{\cR} =\left\{\begin{pmatrix}f_R\\f_P\\e_R\\e_P\end{pmatrix}\in\R^4~\bigg\vert~\begin{bmatrix}1&1\\0&0\end{bmatrix}\begin{pmatrix}f_R\\f_P\end{pmatrix}+\begin{bmatrix}0&0\\1&-1\end{bmatrix}\begin{pmatrix}e_R\\e_P\end{pmatrix}\right\},\ \mathcal L_{\cR}=\R^0,\ \mathcal R_{\cR} =\left\{(f_R,e_R)\in\R^2\vert f_R=-g(e_R)\right\},\]
%with sign-preserving \emph{conductance function} $g\in C(\R,\R)$. From this pH system, one can derive
%\begin{equation*}%\label{Realtion-Con}
%i(t)=g(u(t)),\quad \forall t\in\R.
%\end{equation*}

%\subsubsection{Resistances}
%\noindent A resistance is modelled as a $2$-terminal component $(\mathcal D_{\cR},\mathcal L_{\cR},\mathcal R_{\cR})$ on one edge (Fig.~\ref{2port-to-graph}) whose dynamics
%\begin{equation*}
%(f_R(t),i(t),e_R(t),u(t))\in\mathcal D_{\cR},\ (f_R(t),e_R(t))\in\mathcal R_{\cR},
%\end{equation*}
%for all $t\in\R$, where
%\[
%\mathcal D_{\cR} %=\left\{\begin{pmatrix}f_R\\f_P\\e_R\\e_P\end{pmatrix}\in\R^4~\bigg\vert~\begin{bmatrix}1&1\\0&0\end{bmatrix}\begin{pmatrix}f_R\\f_P\end{pmatrix}+\begin{bmatrix}0&0\\1&-1\end{bmatrix}\begin{pmatrix}e_R\\e_P\end{pmatrix}\right\},
%\]
%\[\mathcal L_{\cR}=\R^0,\quad\mathcal R_{\cR} =\left\{(f_R,e_R)\in\R^2\vert -r(f_R)=e_R\right\},\]
%with sign-preserving \emph{resistance function} $r\in C(\R,\R)$. From this pH system, one can derive
%\begin{equation*}%\label{Realtion-R}
%u(t)=r(i(t)),\quad\forall t\in\R.
%\end{equation*}
\subsubsection{Ideal and PN-junction diodes}

\begin{figure}
	 \centering
	\begin{minipage}[t]{0.45\textwidth}
		\centering
		\begin{circuitikz}%[every loop/.style={<-,shorten <=1pt,min distance=12mm}]
		\draw
			(0,0) to[Do,v^>=$u_{\mathpzc{D}}$,*-*,i>_=$\quad i_{\mathpzc{D}}$] (4,0);
		\end{circuitikz}
		\caption{Circuit symbol of a diode.}
		% \label{D-symbol}
	\end{minipage}\hfill
%	\begin{minipage}[t]{0.45\textwidth}
%		\centering
%		\begin{circuitikz}%[every loop/.style={<-,shorten <=1pt,min distance=12mm}]
%		\draw
%			(0,0) to[Do,v^>=$u_{\mathpzc{J}}$,*-*,i>_=$\quad i_{\mathpzc{J}}$] (4,0);
%		\end{circuitikz}
%		\caption{Circuit symbol of a PN-junction diode.}
%		% \label{J-symbol}
%	\end{minipage}
\end{figure}
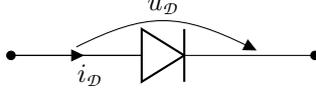

\noindent An {\em ideal diode} is modelled as a two-terminal component $(\mathcal D_{\cDs},\mathcal L_{\cDs},\mathcal R_{\cDs})$ with one port (see Fig.~\ref{2port-to-graph}), and dynamics
\begin{equation*}
(-i_{\cDs}(t),i_{\cDs}(t),u_{\cDs}(t),u_{\cDs}(t))\in\mathcal D_{\cDs},\ (j_{\cDs}(t),\phi_{\cDs}(t))\in\mathcal R_{\cDs},
\end{equation*}
where $\mathcal D_{\cDs}=\mathcal D_{1}$ with $\mathcal D_{1}$ as defined in \eqref{eq:Diracn}, $\mathcal L_{\cD}=\{0\}$ and
\[\mathcal R_{\cD} =\left\{(-i_{\cDs},u_{\cDs})\in\R^2~\vert~ i_{\cDs}u_{\cDs}=0\,\wedge \,i_{\cDs}\leq0\,\wedge \,u_{\cDs}\leq0\right\}.\]
From this pH system, one can derive that
\begin{equation*}%\label{Relation-D}
(i_{\cDs}(t),u_{\cDs}(t))\in\left(\{0\}\times\R_{\leq 0}\right)\cup\left(\R_{\geq0}\times\{0\}\right).
\end{equation*}
A {\em PN-junction diode} is modelled as a one-port component $(\mathcal D_{\cDs},\mathcal L_{\cDs},\mathcal R_{\cDs})$ with $\mathcal D_{\cDs}$ and $\mathcal L_{\cDs}$ as for the ideal diode, and the resistive relation is, for some constants $a,b>0$, given by
\[\mathcal R_{\cD} =\left\{(-i_{\cDs},u_{\cDs})\in\R^2~\vert~i_{\cDs}=a\left(e^{\frac{u_{\cDs}}{b}}-1\right)\right\}.\]
From the dynamics of this pH system, one can derive the characteristic equation \cite[Eq.\ (39.46)]{Kuepf17}%\cite[Sec.~2.1.3]{CircTheo}
\begin{equation*}%\label{Relation-J}
i_{\cDs}(t)=a\left(e^{\frac{u_{\cDs}(t)}{b}}-1\right).
\end{equation*}
The PN-junction diode serves as an approximation for an ideal diode. In a~certain sense, the behavior of a~PN-junction diode indeed tends to that of the ideal diode, if $b\to0$.
\subsubsection{Transformers}

\begin{figure}
	% \centering
	\begin{minipage}[t]{0.45\textwidth}
		\centering
		\begin{circuitikz}[scale=0.75, transform shape, european voltages]
		\draw (0,0) node [transformer core](T){}
		(T.A1) node[fill=black,circle,scale=0.5,anchor=east]{}
		($(T.A1)-(0.2,0)$) node[anchor=east]{$v_1$}
		(T.A2) node[fill=black,circle,scale=0.5,anchor=east]{}
		($(T.A2)-(0.2,0)$) node[anchor=east]{$v_2$}
		(T.B1) node[fill=black,circle,scale=0.5,anchor=west]{}
		($(T.B1)+(0.2,0)$) node[anchor=west]{$v_4$}
		(T.B2) node[fill=black,circle,scale=0.5,anchor=west]{}
		($(T.B2)+(0.2,0)$) node[anchor=west]{$v_3$};
		\end{circuitikz}
		\caption{Circuit symbol of a transformer.}
		% \label{T-symbol}
	\end{minipage}\hfill
	\begin{minipage}[t]{0.45\textwidth}
		\centering
		\includegraphics[scale=0.6]{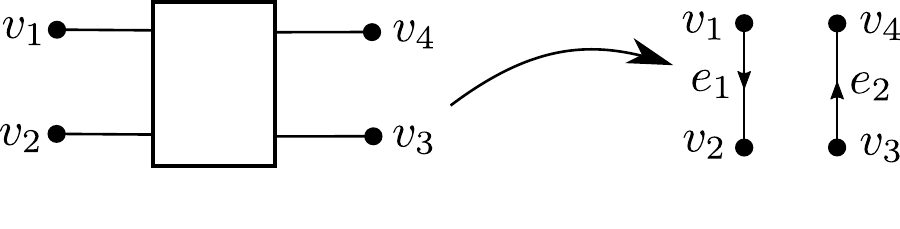}
		\caption{Deriving the underlying graph of a transformer.}
		\label{T-to-graph}
	\end{minipage}
\end{figure}

\noindent A transformer is modelled as a four-terminal component with two ports, see Fig.~\ref{T-to-graph}. It is described by the pH system
$(\mathcal D_{\cT},\mathcal L_{\cT},\mathcal R_{\cT})$, where we use the Dirac structure $\mathcal D_{\cT}=\mathcal D_{2}$ with $\mathcal D_{2}$ as defined in \eqref{eq:Diracn} and trivial Lagrange submanifold $\mathcal L_{\cT}=\{0\}$. The dynamics are given by
\[\begin{aligned}
(-i_{\cT1}(t),-i_{\cT2}(t),i_{\cT1}(t),i_{\cT2}(t),u_{\cT1}(t),u_{\cT2}(t),u_{\cT1}(t),u_{\cT2}(t))\in\mathcal D_{\cT},\\
(-i_{\cT1}(t),-i_{\cT2}(t),u_{\cT1}(t),u_{\cT2}(t))\in\mathcal R_{\cT},
\end{aligned}
\]
with, for some $T\in\R$,
\[
\begin{aligned}
\mathcal R_{\cT} =\left\{(-i_{\cT1},-i_{\cT2},u_{\cT1},u_{\cT2})\in\R^4\;\vert\; Ti_{\cT1}=-i_{\cT2},\ u_{\cT1}=Tu_{\cT2}\right\}.
\end{aligned}\]
From this pH system, one can derive $Ti_{\cT1}(t)=-i_{\cT2}(t)$ and $u_{\cT1}(t)=Tu_{\cT2}(t)$, which means that a~transformer is a~power-conserving component.

\subsubsection{NPN transistors}

\begin{figure}
	% \centering
	\begin{minipage}[t]{0.45\textwidth}
		\centering
		\begin{circuitikz}[scale=0.75, transform shape]
		\draw (0,0) node[npn] (npn) {}
		($(npn.base)-(1,0)$) node[anchor=east] {B}
		($(npn.collector)+(0,0.8)$) node[anchor=west] {C}
		($(npn.emitter)-(0,0.8)$) node[anchor=west] {E};
		\draw ($(npn)-(0.3,0)$) circle [radius=12.1pt];
		\draw ($(npn.base)-(0.8,0)$) node[scale=0.5,circle,fill=black] {} to [short,i=${i_B}$] (npn.base);
		\draw ($(npn.collector)+(0,0.8)$) node[scale=0.5,circle,fill=black] {} to[short,i_=${i_C}$] (npn.collector);
		\draw (npn.emitter) to[short,i=${i_E}$] ($(npn.emitter)-(0,0.8)$) node[scale=0.5,circle,fill=black] {} ;
		\draw ($(npn.base)-(1,0)$) to [open,v^=$u_{BC}$] ($(npn.collector)+(0,0.8)$);
		\draw ($(npn.base)-(1,0)$) to [open,v=$u_{BE}$] ($(npn.emitter)-(0,0.8)$);
		 %\vertex[label=$p_1$]($v_B$) at ($(npn.base)-(0.3,0)$) {};
		\end{circuitikz}
		\caption{Circuit symbol of a NPN transistor.}
		% \label{NPN-symbol}
	\end{minipage}\hfill
	\begin{minipage}[t]{0.45\textwidth}
		\centering
		\includegraphics[scale=0.7]{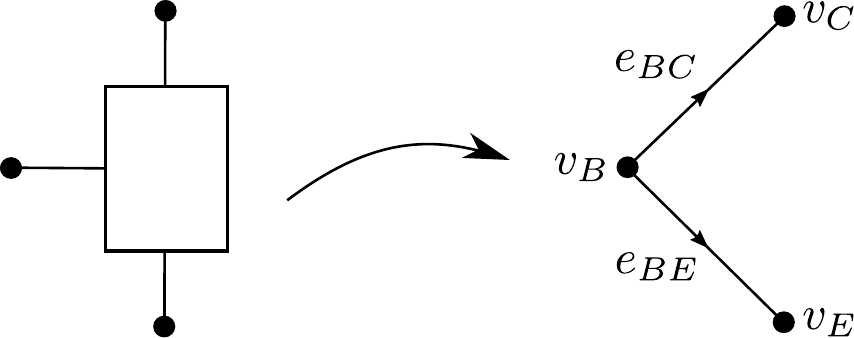}
		\caption{Deriving the underlying graph of an NPN transistor.}
		\label{NPN-to-graph}
	\end{minipage}
\end{figure}

\noindent A~transistor is a~component with three terminals, which are called {\em emitter}, {\em basis} and {\em collector}. We replace this by a~graph with two edges, which are respectively located are between basis and collector, and basis and emitter, see Fig.~\ref{NPN-to-graph}.
The behavior of a transistor {\em of type NPN} is often modelled by the \emph{Ebers-Moll model} \cite[Eqs.\ (5.26) \& (5.27)]{SeSm04}, which can, in a~certain voltage and current range around zero, be summarized by the equations
\begin{equation}
\begin{aligned}
i_C(t)=&~i_S\left(e^{\tfrac{u_{BE}(t)}{V_T}}-1\right)-\tfrac{i_S}{\alpha_R}\left(e^{\tfrac{u_{BC}(t)}{V_T}}-1\right),\\
i_E(t)=&~\tfrac{i_S}{\alpha_F}\left(e^{\tfrac{u_{BE}(t)}{V_T}}-1\right)-i_S\left(e^{\tfrac{u_{BC}(t)}{V_T}}-1\right),
\end{aligned}\label{T1}
\end{equation}
for some constants $\alpha_F\in\left[\tfrac{50}{51},\tfrac{1000}{1001}\right]$, $\alpha_R\in\left[\tfrac{1}{100},\tfrac{1}{2}\right]$, $i_S\in[10^{-15},10^{-12}]$, $V_T\approx\frac{1}{40}$ \cite[pp.~382-394]{SeSm04}. Hereby, $i_C(t)$, $i_E(t)$, $u_{BE}(t)$, $u_{BC}(t)$ respectively denote the collector current, emitter current, basis-emitter voltage and basis collector voltage. Note that, by the Kirchhoff laws, the basis current fulfills $i_B(t)=i_E(t)-i_C(t)$ and the
collector emitter voltage is given by $u_{CE}(t)=u_{BE}(t)-u_{BC}(t)$.
We model an NPN transistor as a `resistive' two-port component $(\mathcal D_{\cN},\mathcal L_{\cN},\mathcal R_{\cN})$ on two edges,
where $\mathcal D_{\cN}=\mathcal D_{2}$ with $\mathcal D_{2}$ as defined in \eqref{eq:Diracn}, $\mathcal L_{\cN}=\{0\}$ and
\[\mathcal R_{\cN} =\left\{(i_{C},-i_{E},u_{BC},u_{BE})\in\R^4\Bigg\vert \begin{array}{l}
i_C=~i_S\left(e^{\tfrac{u_{BE}}{V_T}}-1\right)-\tfrac{i_S}{\alpha_R}\left(e^{\tfrac{u_{BC}}{V_T}}-1\right),\\
i_E=~\tfrac{i_S}{\alpha_F}\left(e^{\tfrac{u_{BE}}{V_T}}-1\right)-i_S\left(e^{\tfrac{u_{BC}}{V_T}}-1\right),
\end{array}\right\}\cap U_0,
\]
where $U_0\subset\R^4$ is a~neighborhood of the origin. The dynamics of the system read
\[(i_{C}(t),-i_{E}(t),-i_C(t),i_E(t),u_{BC}(t),u_{BE}(t),u_{BC}(t),u_{BE}(t))\in\mathcal D_{\cN},\; (i_C(t),-i_E(t),u_{BC}(t),u_{BE}(t))\in\mathcal R_{\cN},\]
which implies \eqref{T1}, at least as long as $(i_C(t),-i_E(t),u_{BC}(t),u_{BE}(t))\in U_0$. Note that we have provided the collector current $i_C(t)$ with another sign, since it is - in contrast to the emitter current and the basis-emitter current - directed contrarily to the basis-collector current.\\
Note that, if we choose $U_0=\R^4$, then the relation $\mathcal R_{\cN}$ is \emph{not} resistive, since for there may exist quadruples $(i_C,-i_E,u_{BC},u_{BE})\in \mathcal R_{\cN}$ holds $i_Cu_{BC}-i_Eu_{BE}>0$. However, we can show that $\mathcal R_{\cN}$ is resistive for a~suitable neighborhood $U_0\subset \R^4$ of the origin. This can be seen as follows: Since for
$(u_{BC},u_{BE})\in (\R\setminus\{0\})^2$ holds
\[\begin{aligned}
&\begin{pmatrix}u_{BC}\\u_{BE}\end{pmatrix}^\top \begin{pmatrix}-\tfrac{i_S}{\alpha_R}\left(e^{\tfrac{u_{BC}}{V_T}}-1\right)+i_S\left(e^{\tfrac{u_{BE}}{V_T}}-1\right)\\i_S\left(e^{\tfrac{u_{BC}}{V_T}}-1\right)-\tfrac{i_S}{\alpha_F}\left(e^{\tfrac{u_{BE}}{V_T}}-1\right)\end{pmatrix}\\
=&\begin{pmatrix}u_{BC}\\u_{BE}\end{pmatrix}^\top \underbrace{\begin{bmatrix}-\tfrac{i_S}{\alpha_R u_{BC}}\left(e^{\tfrac{u_{BC}}{V_T}}-1\right)&\frac{i_S}{u_{BE}}\left(e^{\tfrac{u_{BE}}{V_T}}-1\right)\\\frac{i_S}{u_{BC}}\left(e^{\tfrac{u_{BC}}{V_T}}-1\right)&-\tfrac{i_S}{\alpha_Fu_{BE}}\left(e^{\tfrac{u_{BE}}{V_T}}-1\right)\end{bmatrix}}_{=:A(u_{BC},u_{BE})}
\begin{pmatrix}u_{BC}\\u_{BE}\end{pmatrix}.
\end{aligned}\]
Namely, by using that $A(\cdot,\cdot)$ has a~continuous extension to $\R^2$ with
\[A(0,0)=\frac{i_S}{V_T}\cdot\begin{bmatrix}-\frac{1}{\alpha_R}&1\\1&-\frac{1}{\alpha_F}\end{bmatrix}.\]
By $\alpha_F\in\left[\tfrac{50}{51},\tfrac{1000}{1001}\right]$, $\alpha_R\in\left[\tfrac{1}{100},\tfrac{1}{2}\right]$, we have $\alpha_F\cdot\alpha_R<1$, which leads to negative definiteness of $A(0,0)=\frac12 (A(0,0)+A(0,0)^\top)$. The continuity of $(u_{BC},u_{BE})\mapsto \frac12 (A(u_{BC},u_{BE})+A(u_{BC},u_{BE})^\top)$ implies that there exists some neighborhood $U_0\subset\R^4$ such that this function takes values in the cone of negative definite matrices on $U_0$. This consequences that, by taking this neighborhood $U_0$,  $\mathcal R_{\cN}$ is a~resistive relation.

%and it can be checked that for $e_{R_1}\neq0,~e_{R_2}=0$ we have $e_R^\top f_R>0$. On the other hand, for $e_{R_1}=0,~e_{R_2}\neq0$ we have %$e_R^\top f_R<0$ and for $e_{R_1}=e_{R_2}$ we have $e_R^\top f_R=0$.  In fact, depending on the particular values of $\alpha_F,\alpha_R$, and %$i_S$, one can find certain regions of operation\todo[color=green!40]{Plot?} in which the behavior of these transistors is port-Hamiltonian. A %reason the port-Hamiltonian modelling of such transistors fails, is that, in order to derive the equations (\ref{T1}) and (\ref{T2}), one %considers an equivalent circuit \cite[Fig.~5.8]{SeSm04} containing controlled sources. As we will see next, sources can be modelled in a %port-Hamiltonian fashion. However, these ports associated to the sources are closed by an interconnection which is not port-Hamiltonian. In other %words, the closed-loop system is not port-Hamiltonian.  %no resistive flow or efforts are involved in their (the sources) modelling.

\subsubsection{Current and voltage sources}

\begin{figure}
	% \centering
	\begin{minipage}[t]{0.45\textwidth}
		\centering
		\begin{circuitikz}[european voltages]
		\draw (3,0) to [I,v_>=$u_{\mathpzc{I}}(t)$,*-*,i>^=$\quad i_{\mathpzc{I}}(t)$] (0,0);
		\end{circuitikz}
		\caption{Circuit symbol of a current source.}
		 \label{IS-symbol}
	\end{minipage}\hfill
	\begin{minipage}[t]{0.45\textwidth}
		\centering
		\begin{circuitikz}[european voltages]
		\draw (3,0) to [V,v_>=$u_{\mathpzc{V}}(t)$,*-*,i>^=$\quad i_{\mathpzc{V}}(t)$] (0,0);
		\end{circuitikz}
		\caption{Circuit symbol of a voltage source.}
		 \label{VS-symbol}
	\end{minipage}
\end{figure}

\noindent The sources of the electrical circuit represent the ports of the system, that is points at which physical interaction of the electrical circuit with the environment happens. We may distinguish two types of sources: \emph{current sources} and \emph{voltage sources}, see Fig.~\ref{IS-symbol} and Fig.~\ref{VS-symbol}. The name indicates which physical variable is controlled or influenced by the environment. This variable is also denoted as \emph{input}, while the other is denoted as \emph{output}. However, this distinction is not relevant for the \emph{geometrical} formulation of pH systems (cf. \cite{MMW18}). We unite both classes under the term \emph{sources}. These have two terminals, and, consequently, one port (see Fig.~\ref{2port-to-graph}). Sources are modelled as a~pH system $(\mathcal D_{\cS},\mathcal L_{\cS},\mathcal R_{\cS})$, where the Dirac structure is $\mathcal D_{\cS}=\mathcal D_{1}$ with $\mathcal D_{1}$ as defined in \eqref{eq:Diracn}, and the Lagrange submanifold and resistive relation are trivial, i.e., $\mathcal L_{\cS}=\mathcal R_{\cS}=\{0\}$. The dynamics are
\begin{equation*}
(-i_{\cS}(t),i_{\cS}(t),u_{\cS}(t),u_{\cS}(t))\in\mathcal D_{\cS}.
\end{equation*}

\begin{example}[AC/DC converter]
We illustrate our methodology by considering an AC/DC converter, which we model by the electrical circuit shown in Fig.\ \ref{ACDC-circ}. The AC/DC converter consists of a source $\mathpzc{S}=(\mathcal D_{\cS},\mathcal L_{\cS},\mathcal R_{\cS})$, a transformer $\mathpzc{T}=(\mathcal D_{\cT},\mathcal L_{\cT},\mathcal R_{\cT})$, four PN-junction diodes $\mathpzc{J}_i=(\mathcal D_{\cJ_i},\mathcal L_{\cJ_i},\mathcal R_{\cJ_i})$ for $i\in\{1,...,4\}$, a capacitor $\mathpzc{C}=(\mathcal D_{\cC},\mathcal L_{\cC},\mathcal R_{\cC})$, and a `sink' $\mathpzc{O}=(\mathcal D_{\cO},\mathcal L_{\cO},\mathcal R_{\cO})$ (modelled like a source), which are connected by the vertices $v_1,...,v_6$ as shown in Fig.~\ref{ACDC-to-graph}. The circuit graph $\mathcal{G}=(V,E,\text{init},\text{ter})$ with $V=\{v_1,...,v_6\}$ and  $E=\{e_1,...,e_9\}$ has two components, and we ground the nodes below the voltage source and the capacitance, i.e., we choose $S=\{v_2,v_3\}$.
Let $A\in\R^{4\times9}$ be obtained from the incidence matrix of $\mathcal{G}$ by deleting the rows corresponding to the grounded nodes. We arrive at a pH system $(\mathcal D,\mathcal L,\mathcal R)$ as in \eqref{circ-model}, whose dynamics read
\[
\left(\ddt\begin{pmatrix}-q_1\\-q_4\\-q_5\\-q_6\\-q_{\cC}\end{pmatrix},\begin{pmatrix}-i_{\cT 1}\\-i_{\cT 2}\\-i_{\cD 1}\\-i_{\cD 2}\\-i_{\cD 3}\\-i_{\cD 4}\end{pmatrix},\begin{pmatrix}-i_{\cV}\\-i_{\cO}\end{pmatrix},\begin{pmatrix}{\phi_1}\\{\phi_4}\\{\phi_5}\\{\phi_6}\\u_{\cC}\end{pmatrix},\begin{pmatrix}u_{\cT 1}\\u_{\cT 2}\\u_{\cD 1}\\u_{\cD 2}\\u_{\cD 3}\\u_{\cD 4}\end{pmatrix},\begin{pmatrix}u_{\cV}\\u_{\cO}\end{pmatrix}\right)\in\mathcal D,
\]
\[
\left(\begin{pmatrix}-q_1\\-q_4\\-q_5\\-q_6\\-q_{\cC}\end{pmatrix},\begin{pmatrix}{\phi_1}\\{\phi_4}\\{\phi_5}\\{\phi_6}\\u_{\cC}\end{pmatrix}\right)\in\mathcal L,\quad\quad \left(\begin{pmatrix}-i_{\cT 1}\\-i_{\cT 2}\\-i_{\cD 1}\\-i_{\cD 2}\\-i_{\cD 3}\\-i_{\cD 4}\end{pmatrix},\begin{pmatrix}u_{\cT 1}\\u_{\cT 2}\\u_{\cD 1}\\u_{\cD 2}\\u_{\cD 3}\\u_{\cD 4}\end{pmatrix}\right)\in\mathcal R.
\]

% \[A_0=\begin{bmatrix}-1&1&0&0&0&0&0&0&0\\1&-1&0&0&0&0&0&0&0\\0&0&0&1&0&0&1&-1&-1\\0&0&1&-1&1&0&0&0&0\\0&0&0&0&-1&-1&0&1&1\\0&0&-1&0&0&1&-1&0&0\end{bmatrix}.\]

\begin{figure}
	% \centering
	\begin{minipage}[t]{0.4\textwidth}
		\centering
		% \begin{circuitikz}[scale = 0.8]
		% \ctikzset{quadpoles/transformer core/width=2, quadpoles/transformer core/height=2}
		% \ctikzset{quadpoles style=inline}
		% % \ctikzset{quadpoles/transformer core/inner=1, quadpoles/transformer core/width=0.6}
		% \draw (0,0) node [transformer core](T){};
		% \draw ($(T.B1)!0.5!(T.B2)+(0.5,0)$) node[](W){};
		% \draw ($(T.B1)+(2.05,0)$) node[](N){};
		% \draw ($(T.B2)+(2.05,0)$) node[](S){};
		% \draw ($(W)+(3.1,0)$) node[](E){};
		% %       (T.A1) node[above]
		% %       (T.A2) node[below]
		% %       (T.B1) node[above]
		% %       (T.B2) node[below] {B2}
		% %       (T.base) node{K};

		% \draw (T.A1) to [] ($(T.A1)-(2,0)$) to [V,i=${}$,v_<=${}$] ($(T.A2)-(2,0)$) to (T.A2);
		% \draw (T.B1) to [] ($(T.B1)+(2.05,0)$);
		% \draw (T.B2) to [] ($(T.B2)+(2.05,0)$);
		% \draw (W) to [D,/tikz/circuitikz/bipoles/length=0.9cm,i_=${}$] (N) to [D,/tikz/circuitikz/bipoles/length=0.9cm,i_=${}$] (E) to ($(E)+(1,0)$);
		% \draw (W) to [D,/tikz/circuitikz/bipoles/length=0.9cm,i_=${}$] (S) to [D,/tikz/circuitikz/bipoles/length=0.9cm,i_=${}$] (E);
		% \draw ($(T.B1)!0.5!(T.B2)+(0.5,0)$) to ($(W)+(0,-2)$) to ($(E)+(0,-2)$) to ($(E)+(1,-2)$);
		% \draw (E) to ($(E)+(1.0,0)$);
		% \draw (E) to [C,/tikz/circuitikz/bipoles/length=0.9cm,i_=${}$] ($(E)+(0,-2)$);
		% \end{circuitikz}
		\includegraphics[width=\textwidth]{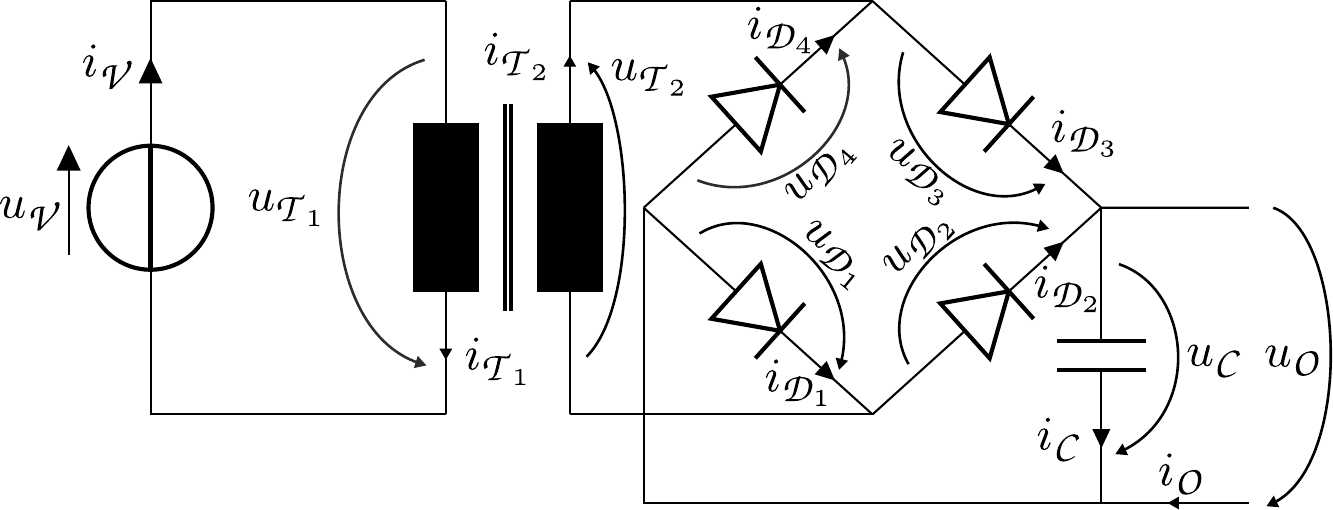}
		\caption{AC/DC converter circuit}
		\label{ACDC-circ}
	\end{minipage}\hfill
	\begin{minipage}[t]{0.55\textwidth}
		\centering
		\includegraphics[width=\textwidth]{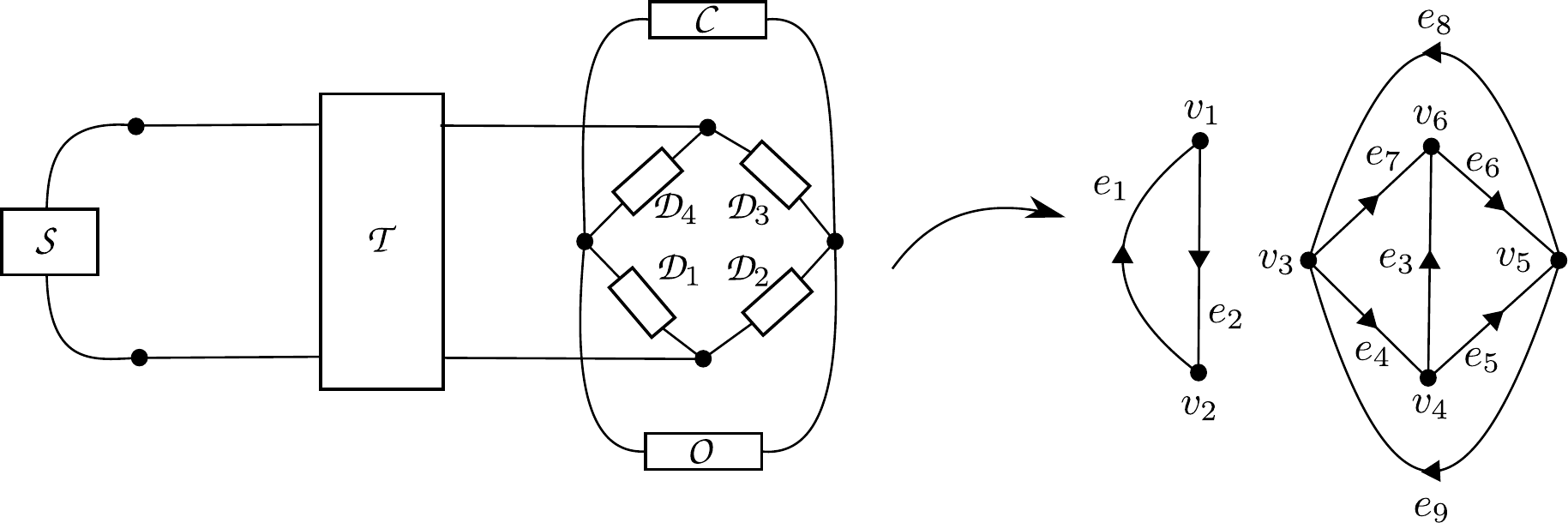}
	\caption{Obtaining the underlying graph of the AC/DC converter}
	\label{ACDC-to-graph}
	\end{minipage}
\end{figure}
\end{example}

\section{Comparison with other formulations of electrical circuits}
\noindent With the attention electrical circuits attracted over the past decades, quite a bunch of \lq standard formulations\rq\ of the dynamics have emerged. An overview of popular models in the context of DAEs is found in \cite{Riaz13}. We compare for certain electrical circuits the dynamics of our port-Hamiltonian modelling (\ref{circ-model}) with other equations used in the modelling of electrical circuits.
\subsection{The (charge/flux-oriented) modified nodal analysis}\label{subsec-MNA}
\noindent Let an electrical circuit consisting of
conductances, inductances, capacitances and sources.
Let
\begin{align*}
&(\mathcal D_{\cR_i},\mathcal L_{\cR_i},\mathcal R_{\cR_i})_{i\in\{1,\ldots,l_{\cR}\}}, &&\!\!\!\!\!\!\!\!\!\!\!\!\!\!\!\!\!\!\!\!\!\!\!\!\!\!\!\!\!\!\!\!\!(\mathcal D_{\cL_i},\mathcal L_{\cL_i},\mathcal R_{\cL_i})_{i\in\{1,\ldots,l_{\cL}\}},\notag\\
&(\mathcal D_{\cC_i},\mathcal L_{\cC_i},\mathcal R_{\cC_i})_{i\in\{1,\ldots,l_{\cC}\}}, &&\!\!\!\!\!\!\!\!\!\!\!\!\!\!\!\!\!\!\!\!\!\!\!\!\!\!\!\!\!\!\!\!\!(\mathcal D_{\cS_i},\mathcal L_{\cS_i},\mathcal R_{\cS_i})_{i\in\{1,\ldots,l_{\cS}\}}.%\label{Elements}
\end{align*}
be the pH systems modelling the components as derived in Section \ref{subsec-comp}. Let $\ell_{p,\cR i}$ be the number of ports of the component modelled by $(\mathcal D_{\cR_i},\mathcal L_{\cR_i},\mathcal R_{\cR_i})$, and let $\ell_{p,\cL i}$ and $\ell_{p,\cC i}$ be analogously defined. Moreover, let
\[m_{\cR}=\sum_{i=1}^{l_{\cR}}\ell_{p,\cR i},\quad
m_{\cL}=\sum_{i=1}^{l_{\cR}}\ell_{p,\cL i},\quad
m_{\cC}=\sum_{i=1}^{l_{\cC}}\ell_{p,\cR i},\quad m_{\cS}=l_{\cS},\]
and introduce
\[
i_{\cR}=\begin{pmatrix}i_{\cR 1}\\\vdots\\i_{\cR {m_{\cR}}}\end{pmatrix},\;\; i_{\cL}=\begin{pmatrix}i_{\cL 1}\\\vdots\\i_{\cL {m_{\cL}}}\end{pmatrix},\;\; i_{\cC}=\begin{pmatrix}i_{\cC 1}\\\vdots\\i_{\cC {m_{\cC}}}\end{pmatrix},\;\; i_{\cS}=\begin{pmatrix}i_{\cS 1}\\\vdots\\i_{\cS {m_{\cS}}}\end{pmatrix},\;\; i=\begin{pmatrix}i_{\cR}\\i_{\cL}\\i_{\cC}\\i_{\cS}\end{pmatrix},
\]
and analogous notations for $u_{\cR}$, $u_{\cL}$, $u_{\cC}$, $u_{\cS}$, $u$, as well as
\[
q_{\cC}=\begin{pmatrix}q_{\cC 1}\\\vdots\\q_{\cC {m_{\cC}}}\end{pmatrix},\quad \psi_{\cL}=\begin{pmatrix}\psi_{{\cL}1}\\\vdots\\\psi_{{\cL}m_{\cL}}\end{pmatrix},\quad g(u_{\cR})=\begin{pmatrix}g_1(u_{\cR 1})\\\vdots\\g_{m_{\cR}}(u_{{\cR}m_{\cR}})\end{pmatrix},\]
\[
H_{\cC}(q_{\cC})=\sum_{i=1}^{m_{\cC}}H_{\cC i}(q_{\cC i}),\quad H_{\cL}(\psi_{\cL})=\sum_{i=1}^{m_{\cL}} H_{\cL i}(\psi_{{\cL}i}).
\]
Further, let $\mathcal{G}=(V,E,\text{init},\text{ter})$ be the graph induced by the electrical circuit with $|V|=n$ and $|E|=m$.
%It follows from Section \ref{subsec-comp} that , since each of the present components is modelled on one edge, i.e., each edge $e\in E$ represents either a resistance, an inductance, a capacitor or a source.
Let $S$ be the set of grounded vertices (cf.\ Definition~\ref{def-K}), and let  $A\in\R^{(n-|S|)\times m}$ be obtained from the incidence matrix of $\mathcal{G}$ by deleting the rows corresponding to the vertices in $S$.
%The leads us to the decomposition of the edge set as $E_{\cR}\ddt\cup E_{\cL}\ddt\cup E_{\cC}\ddt\cup E_{\cD}\ddt\cup E_{\cS}$ with $|E_R|=m_R,\ |E_L|=m_L,\ |E_C|=m_C,\ |E_I|=m_I,\ |E_V|=m_V$.
By a~suitable reordering, we may sort into edges to the specific components, i.e.,
\begin{equation*}%\label{decomp-A}
A=\begin{bmatrix}A_{\cR}&A_{\cL}&A_{\cC}&A_{\cS}\end{bmatrix},
\end{equation*}
where the columns of $A_{\cR}\in\R^{(n-|S|)\times m_{\cR}}$, $A_{\cL}\in\R^{(n-|S|)\times m_{\cL}}$, $A_{\cC}\in\R^{(n-|S|)\times m_{\cC}}$ and $A_{\cS}\in\R^{(n-|S|)\times m_{\cS}}$ respectively represent the edges corresponding to conductances, inductances, capacitances and sources.
For the representation of the port-Hamiltonian dynamics of the electrical circuit, first note that the Dirac structure of the pH system
\[\bigtimes_{i=1}^{m_{\cR}} (\mathcal D_{\cR_i},\mathcal L_{\cR_i},\mathcal R_{\cR_i})\times\bigtimes_{i=1}^{m_{\cL}}\mathcal (\mathcal D_{\cL_i},\mathcal L_{\cL_i},\mathcal R_{\cL_i})\times\bigtimes_{i=1}^{m_{\cC}}\mathcal (\mathcal D_{\cC_i},\mathcal L_{\cC_i},\mathcal R_{\cC_i})\times\bigtimes_{i=1}^{m_{\cS}}\mathcal (\mathcal D_{\cS_i},\mathcal L_{\cS_i},\mathcal R_{\cS_i})\]
is given by
\begin{align*}
\mathcal D_{\rm prod}
&=\Bigg\{\begin{pmatrix}-i_{\cL},-i_{\cC},-i_{\cR},i_{\cR},i_{\cL},i_{\cC},-i_{\cS},i_{\cS},i_{\cL},u_{\cC},u_{\cR},u_{\cR},u_{\cL},u_{\cC},u_{\cS},u_{\cS}\end{pmatrix}\in\R^{2m}\times\R^{2m}~\vert\\[-0.3cm]
&\quad\quad\quad\quad\quad\quad\quad\quad\quad\quad\quad\quad i_{\cL},u_{\cL}\in\R^{m_{\cL}},\ i_{\cC},u_{\cC}\in\R^{m_{\cC}},\  i_{\cR},u_{\cR}\in\R^{m_{\cR}},\ i_{\cS},u_{\cS}\in\R^{m_{\cS}}\bigg\}
\end{align*}
and
\[\begin{aligned}
\mathcal D^S_K(\mathcal{G})=\Bigg\{\begin{pmatrix}j,i_{\cR},i_{\cL},i_{\cC},i_{\cS},\phi, u_{\cR},u_{\cL},u_{\cC},u_{\cS}\end{pmatrix}\in\R^{n-|S|}\times\R^m\times\R^{n-|S|}\times\R^m~\big\vert\quad\quad\quad\quad\quad\quad\quad\\\left.\begin{bmatrix}I&A_{\cR}&A_{\cL}&A_{\cC}&A_{\cS}\\0&0&0&0&0\\0&0&0&0&0\\0&0&0&0&0\\0&0&0&0&0\end{bmatrix}\begin{pmatrix}j\\i_{\cR}\\i_{\cL}\\i_{\cC}\\i_{\cS}\end{pmatrix}+\begin{bmatrix}0&0&0&0&0\\-A_{\cR}^\top&I&0&0&0\\-A_{\cL}^\top&0&I&0&0\\-A_{\cC}^\top&0&0&I&0\\-A_{\cS}^\top&0&0&0&I\end{bmatrix}\begin{pmatrix}\phi\\u_{\cR}\\u_{\cL}\\u_{\cC}\\u_{\cS}\end{pmatrix}=0\right\}.
\end{aligned}\]
It follows that the Dirac structure of
\[(\mathcal D^S_K(\mathcal{G}),\mathcal L^S_K(\mathcal{G}),\{0\})\circ\left(\bigtimes_{i=1}^N(\mathcal D_i,\mathcal L_i,\mathcal R_i)\right)\]
is given by
\begin{subequations}\label{RLC-pH}
\begin{equation}\label{RLC-Dirac}
\begin{aligned}\mathcal D=\Bigg\{&\begin{pmatrix}j,-u_{\cL},-i_{\cC},-i_{\cR},-i_{\cS},\phi,i_{\cL},u_{\cC},u_{\cR},u_{\cS}\end{pmatrix}\in\R^{n-|S|}\times\R^m\times\R^{n-|S|}\times\R^m~\big\vert\quad\quad\quad\quad\quad\quad\quad\\&\left.\begin{bmatrix}I&0&A_{\cC}&A_{\cR}&A_{\cS}\\0&0&0&0&0\\0&-I&0&0&0\\0&0&0&0&0\\0&0&0&0&0\end{bmatrix}\begin{pmatrix}j\\-u_{\cL}\\-i_{\cC}\\-i_{\cR}\\-i_{\cS}\end{pmatrix}+\begin{bmatrix}0&-A_{\cL}&0&0&0\\-A_{\cR}^\top&0&0&I&0\\-A_{\cL}^\top&0&0&0&0\\-A_{\cC}^\top&0&I&0&0\\-A_{\cS}^\top&0&0&0&I\end{bmatrix}\begin{pmatrix}\phi\\i_{\cL}\\u_{\cC}\\u_{\cR}\\u_{\cS}\end{pmatrix}=0\right\},
\end{aligned}
\end{equation}
whereas the Lagrange submanifold and resistive relation read
\begin{align}
\mathcal{L}&=
\Bigg\{\begin{pmatrix}q,\psi_{\cL},q_{\cC},\phi,i_{\cL},u_{\cC}\end{pmatrix}\in\R^{n-|S|}\times\R^{m_{\cL}}\times\R^{m_{\cC}}\times\R^{n-|S|}\times\R^{m_{\cL}}\times\R^{m_{\cC}} ~\Big\vert\nonumber\\[-3mm]&\qquad\qquad\left.q=0\,\wedge\,i_{\cL}=\nabla H_{\cL}(\psi_{\cL})\,\wedge\,u_{\cC}=\nabla H_{\cC}(q_{\cC})\right.\Bigg\},\\[0mm]
\mathcal{R}&=\Big\{(-i_{\cR},u_{\cR})\in\R^{m_{\cR}}\times\R^{m_{\cR}}\;\vert\; i_{\cR}=g(u_{\cR})\Big\}.
\end{align}
\end{subequations}
The triple $(\mathcal D,\mathcal L,\mathcal R)$
with $\mathcal D$, $\mathcal L$ and $\mathcal R$ as in \eqref{RLC-pH} is the port-Hamiltonian representation of a~circuit with conductances, inductances, capacitances and sources in a compact form.
The dynamics of $(\mathcal{D},\mathcal{L},\mathcal{R})$ read
\begin{align*}
&(-\ddt q(t),-\ddt\psi_{\cL}(t),-\ddt q_{\cC}(t),-i_{\cR}(t),-i_{\cS}(t),\phi(t),i_{\cL}(t),u_{\cC}(t),u_{\cR}(t),u_{\cS}(t))\in\mathcal D,\\ &(q(t),\psi_{\cL}(t),q_{\cC}(t),\phi(t),i_{\cL}(t),u_{\cC}(t))\in\mathcal L,\quad (-i_{\cR}(t),e_{\cR}(t))\in\mathcal R,
\end{align*}
which is equivalent to
\[\begin{aligned}
&\begin{bmatrix}I&0&A_{\cC}&A_{\cR}&A_{\cS}\\0&0&0&0&0\\0&-I&0&0&0\\0&0&0&0&0\\0&0&0&0&0\end{bmatrix}\begin{pmatrix}-\ddt q(t)\\-\ddt \psi_{\cL}(t)\\-\ddt q_{\cC}(t)\\-i_{\cR}(t)\\-i_{\cS}(t)\\\end{pmatrix}+\begin{bmatrix}0&-A_{\cL}&0&0&0\\-A_{\cR}^\top&0&0&I&0\\-A_{\cL}^\top&0&0&0&0\\-A_{\cC}^\top&0&I&0&0\\-A_{\cS}^\top&0&0&0&I\end{bmatrix}\begin{pmatrix}\phi(t)\\i_{\cL}(t)\\u_{\cC}(t)\\u_{\cR}(t)\\u_{\cS}(t)\end{pmatrix}=0,\\[2mm]
&q(t)=0,\;\; i_{\cL}(t)=\nabla H_{\cL}(\psi_{\cL}(t)),\;\; u_{\cC}(t)=\nabla H_{\cC}(q_{\cC}(t)),\;\; i_{\cR}(t)=g(u_{\cR}(t)).
\end{aligned}\]
Plugging in the latter relations, we obtain
\begin{equation}
\begin{aligned}
                    A_{\cC}\ddt q_{\cC}(t)+A_{\cR}g(A_{\cR}^\top\phi(t))+A_{\cL}i_{\cL}(t)+A_{\cS}i_{\cS}(t)&=0,\\
					-A_{\cL}^\top\phi(t)+\ddt\psi_{\cL}(t)&=0,\\
					-A_{\cS}^\top\phi(t)+u_{\cS}(t)&=0,\\
					A_{\cC}^\top\phi(t)-\nabla H_{\cC}(q_{\cC}(t))&=0,\\
					i_{\cL}(t)-\nabla H_{\cL}(\psi(t))&=0.
\end{aligned}\label{unsers}
\end{equation}
If we additionally assume that $\nabla H_{\cC}\in C^1(\R^{m_{\cC}},\R^{m_{\cC}})$, $\nabla H_{\cL}\in C^1(\R^{m_{\cL}},\R^{m_{\cL}})$ are homeomorphisms, we can introduce the inverse functions $Q_{\cC}:=(\nabla H_{\cC})^{-1}\in C(\R^{m_{\cC}},\R^{m_{\cC}})$, $\Psi_{\cL}:=(\nabla H_{\cL})^{-1}\in C(\R^{m_{\cL}},\R^{m_{\cL}})$. Then \eqref{unsers} leads to
$q_{\cC}(t)=Q_{\cC}(u_{\cC}(t))$ and $\psi_{\cL}(t)=\Psi_{\cL}(i_{\cL}(t))$. Further decomposing
\[A_{\cS}=\begin{bmatrix}A_{\cI}&A_{\cV}\end{bmatrix},\quad
u_{\cS}=\begin{pmatrix}u_{\cI}\\u_{\cV}\end{pmatrix},\quad
i_{\cS}=\begin{pmatrix}i_{\cI}\\i_{\cV}\end{pmatrix}\]
into edges, voltages and currents to current and voltage sources, we see that \eqref{unsers} leads to the so-called \emph{charge/flux-oriented modified nodal analysis} \cite[Eq.\ (3.21)]{Baec07}
\begin{equation}
\begin{aligned}
A_{\cC}\ddt q_{\cC}(t)+A_{\cR}g(u_{\cR}(t))+A_{\cL}i_{\cL}(t)+A_{\cI}i_{\cI}(t)+A_{\cV}i_{\cV}(t)=&\,0,\\
-A_{\cL}^\top \phi(t)+\ddt \psi_{\cL}(t)=&\,0,\\
-A_{\cV}^\top \phi(t)+u_{\cV}(t)=&\,0,\\
q_{\cC}(t)-Q_{\cC}(A_{\cC}^\top\phi(t))=&\,0,\\
\psi_{\cL}(t)-\Psi_{\cL}(i_{\cL}(t))=&\,0.
\end{aligned}\tag{MNA c/f}\label{MNA c/f}
\end{equation}
If we additionally assume that
$Q_{\cC}\in C(\R^{m_{\cC}},\R^{m_{\cC}})$ and $\Psi_{\cL}\in C^1(\R^{m_{\cL}},\R^{m_{\cL}})$, then we can, by denoting the Jacobians  by
$\cCl(u_{\cC})=\tfrac{\mathrm{d}}{\mathrm{d}u_{\cC}}Q_{\cC}(u_{\cC})$ and $\cLl(i_{\cL})=\tfrac{\mathrm{d}}{\mathrm{d}i_{\cL}}\Psi_{\cL}(i_{\cL})$, reformulate \eqref{MNA c/f} to obtain the \emph{modified nodal analysis} \cite[Eq.\ (52)]{Rei14}
\begin{equation}
\begin{aligned}
 A_{\cC}\cCl(A_{\cC}^\top \phi(t))A_{\cC}^\top \ddt{\phi}(t)+A_{\cR}g(A_{\cR}^\top \phi(t))+A_{\cL}i_{\cL}(t)+A_{\cI}i_{\cI}(t)+A_{\cV}i_{\cV}(t)=&\,0,\\
-A_{\cL}^\top \phi(t)+\cLl(i_{\cL}(t))\ddt{i_{\cL}}(t)=&\,0,\\
-A_{\cV}^\top \phi(t)+u_{\cV}(t)=&\,0.
\end{aligned}\tag{MNA}\label{MNA}
\end{equation}
Note that, if $H_{\cC}\in C^2(\R^{m_{\cC}},\R)$, $H_{\cL}\in C^2(\R^{m_{\cL}},\R)$, then $\cCl(u_{\cC})$ and $\cLl(i_{\cL})$ are, respectively, the inverses of the Hessians of $H_{\cC}$ and $H_{\cL}$ at $Q_{\cC}(u_{\cC})$ and $\Psi_{\cL}(i_{\cL})$.

\subsection{The (charge/flux-oriented) modified loop analysis}
\noindent We present an alternative modelling involving the pH system $(\mathcal D'_K(\mathcal{G}),\mathcal L'_K(\mathcal{G}),\{0\})$ with $\mathcal D'_K(\mathcal{G})$ and $\mathcal L'_K(\mathcal{G})$ as in \eqref{KDS'}. That is, the loops in the underlying graph structure is now taken to model the Kirchhoff laws.
First note that the external flows and efforts variables in the pH system $(\mathcal D^S_K(\mathcal{G}),\mathcal L^S_K(\mathcal{G}),\{0\})$ in Remark~\ref{rem-inv} are, respectively, the current and the voltage of the components, while the external flows and efforts variables in $(\mathcal D'_K(\mathcal{G}),\mathcal L'_K(\mathcal{G}),\{0\})$ are, respectively, the voltage and the current of the components. This means that in order to obtain a pH system $(\mathcal D',\mathcal L',\mathcal R')$ describing the circuit dynamics by performing an interconnection of $(\mathcal D'_K(\mathcal{G}),\mathcal L'_K(\mathcal{G}),\{0\})$ with $N\in\N$ electrical components $(\mathcal D_i,\mathcal L_i,\mathcal R_i)_{i\in\{1,\ldots, N\}}$, i.e.,
\begin{equation*}%\label{circ-model'}
(\mathcal D',\mathcal L',\mathcal R')\coloneqq(\mathcal D'_K(\mathcal{G}),\mathcal L'_K(\mathcal{G}),\{0\})\circ\left(\bigtimes_{i=1}^N(\mathcal D_i,\mathcal L_i,\mathcal R_i)\right),
\end{equation*}
we have to adjust the definition of the components by interchanging the role of the effort and flow variables, which is possible by an argument similar to one in Remark \ref{rem-inv}. Given an electrical circuit consisting of resistances, inductances, capacitances and sources, it can, completely analogous to Section \ref{subsec-MNA}, be shown that the dynamics of $(\mathcal D',\mathcal L',\mathcal R')$
lead, under certain additional invertibility and smoothness assumptions on the functions representing capacitances and inductances, to the \emph{modified loop analysis} \cite[Eq.\ (53)]{Rei14}
\begin{equation*}
\begin{aligned}
B_{\cL}\cLl(B_{\cL}^\top\iota(t))B_{\cL}^\top\ddt{\iota}(t)+B_{\cR}r(B_{\cR}^\top\iota(t))+B_{\cC}u_{\cC}(t)+B_{\cI}u_{\cI}(t)+B_{\cV}u_{\cV}(t)=&\,0,\\
-B_{\cC}^\top\iota(t)+\cCl(u_{\cC}(t))\ddt{u_{\cC}}(t)=&\,0,\\
-B_{\cI}^\top\iota(t)+i_{\cI}(t)=&\,0.
\end{aligned}%\tag{MLA}\label{MLA}
\end{equation*}

\bibliographystyle{plain}
\bibliography{pH-circuits}

\begin{thebibliography}{10}

\bibitem{And91}
B.~Andr\'asfai.
\newblock {\em Graph Theory: Flows, Matrices}.
\newblock Taylor \& Francis, New York London, 1991.

\bibitem{Baec07}
S.~B\"{a}chle.
\newblock {\em Numerical Solution of Differential-Algebraic Systems Arising in
  Circuit Simulation}.
\newblock PhD thesis, Fakult\"at II - Mathematik und Naturwissenschaften,
  Technische Universit{\"a}t Berlin, Berlin, Germany, 2007.

\bibitem{BCGM18}
M.~Barbero-Li{\~n}\'an, H.~Cendra, E.~Garc\'ia-Tora{\~{n}}o Andr\'es, and
  D.~Mart\'in de~Diego.
\newblock New insights in the geometry and interconnection of
  port-{H}amiltonian systems.
\newblock {\em J. Phys. A}, 51(37):375201 (30 p.), 2018.

\bibitem{BMXZ18}
C.~Beattie, V.~Mehrmann, H.~Xu, and H.~Zwart.
\newblock Linear port-{H}amiltonian descriptor systems.
\newblock {\em Math.\ Control Signals Syst.}, 30(4), 2018.
\newblock Article: 17.

\bibitem{BKvdSZ10}
J.~Behrndt, M.~Kurula, A.J. {van der Schaft}, and H.~Zwart.
\newblock {D}irac structures and their composition on {H}ilbert spaces.
\newblock {\em J.\ Math.\ Anal.\ Appl.}, 372(2):402--422, 2010.

\bibitem{BMvdS95}
P.C. Breedveld, B.~Maschke, and A.J. van~der Schaft.
\newblock An intrinsic hamiltonian formulation of the dynamics of
  {LC}-circuits.
\newblock {\em IEEE Trans.\ Circuits Syst.\ I. Regul.\ Pap.}, 42(2):73--82,
  1995.

\bibitem{CvdSB07}
J.~Cervera, A.J. van~der Schaft, and A.~Ba{\~{n}}os.
\newblock Interconnection of port-{H}amiltonian systems and composition of
  dirac structures.
\newblock {\em Automatica}, 43(2):212--225, 2007.

\bibitem{Cou90}
T.J. Courant.
\newblock Dirac manifolds.
\newblock {\em Trans. Amer. Math. Soc.}, 319:631--661, 1990.

\bibitem{Dies17}
R.~Diestel.
\newblock {\em Graph theory}, volume 173 of {\em Graduate Texts in
  Mathematics}.
\newblock Springer, Berlin, 5th edition, 2017.

\bibitem{JvdS14}
D.~Jeltsema and A.J. van~der Schaft.
\newblock Port-{H}amiltonian systems theory: An introductory overview.
\newblock {\em Foundations and Trends in Systems and Control}, 1(2-3):173--387,
  2014.

\bibitem{Lee12}
J.M. Lee.
\newblock {\em Introduction to Smooth Manifolds}, volume 218 of {\em Graduate
  Texts in Mathematics}.
\newblock Springer, New York, 2nd edition, 2012.

\bibitem{MvdS18}
B.~Maschke and A.J. van~der Schaft.
\newblock Generalized port-{H}amiltonian {DAE} systems.
\newblock {\em Systems \& Control Letters}, 121:31--37, 2018.

\bibitem{MvdS19}
B.~Maschke and A.J. van~der Schaft.
\newblock {D}irac and {L}agrange algebraic constraints in nonlinear
  port-{H}amiltonian systems.
\newblock Technical report, 2019.

\bibitem{Kuepf17}
W.~Mathis and A.~Reibiger.
\newblock {\em K\"upfm\"uller Theoretische Elektrotechnik}.
\newblock Springer Vieweg, Berlin, 20th edition, 2017.

\bibitem{MMW18}
C.~Mehl, V.~Mehrmann, and M.~Wojtylak.
\newblock Linear algebra properties of dissipative {H}amiltonian descriptor
  systems.
\newblock {\em SIAM J.\ Matrix Anal.\ Appl.}, 39(3):1489--1519, 2018.

\bibitem{MvdS04a}
C.~Melchiorri and A.J. {van der Schaft}.
\newblock Port {H}amiltonian formulation of infinite dimensional systems - i.
  modeling.
\newblock In {\em 43rd IEEE Conference on Decision and Control, December 14-17,
  2004, Atlantis, Paradise Island, Bahamas}, 2004.

\bibitem{MvdS04b}
C.~Melchiorri and A.J. {van der Schaft}.
\newblock Port {H}amiltonian formulation of infinite dimensional systems - ii.
  boundary control by interconnection.
\newblock In {\em 43rd IEEE Conference on Decision and Control, December 14-17,
  2004, Atlantis, Paradise Island, Bahamas}, 2004.

\bibitem{Rei14}
T.~Reis.
\newblock Mathematical modeling and analysis of nonlinear time-invariant {RLC}
  circuits.
\newblock In P.~Benner, R.~Findeisen, D.~Flockerzi, U.~Reichl, and
  K.~Sundmacher, editors, {\em Large-Scale Networks in Engineering and Life
  Sciences}, Modeling and Simulation in Science, Engineering and Technology,
  pages 125--198. Birkh\"auser, Basel, 2014.

\bibitem{Riaz13}
R.~Riaza.
\newblock {DAE}s in circuit modelling: A survey.
\newblock In Achim Ilchmann and Timo Reis, editors, {\em Surveys in
  Differential-Algebraic Equations I}, Differential-Algebraic Equations Forum,
  pages 97--136. Springer, Berlin-Heidelberg, 2013.

\bibitem{SeSm04}
A.S. Sedra and K.C. Smith.
\newblock {\em Microelectronic Circuits}.
\newblock Oxford University Press, New York, 5th edition, 2004.

\bibitem{vdS10}
A.J. van~der Schaft.
\newblock Characterization and partial synthesis of the behavior of resistive
  circuits at their terminals.
\newblock {\em Systems \& Control Letters}, 59(7):423--428, 2010.

\bibitem{vdS13}
A.J. van~der Schaft.
\newblock Port-{H}amiltonian differential-algebraic systems.
\newblock In A.~Ilchmann and T.~Reis, editors, {\em Surveys in
  Differential-Algebraic Equations I}, Differential-Algebraic Equations Forum,
  pages 173--226. Springer, Berlin Heidelberg, 2013.

\bibitem{vdS17}
A.J. van~der Schaft.
\newblock {\em $L\sb 2$-Gain and Passivity Techniques in Nonlinear Control}.
\newblock Lecture Notes in Control and Information Sciences. Springer, London,
  3rd edition, 2017.

\bibitem{vdSM13a}
A.J. van~der Schaft and B.~Maschke.
\newblock Port-{H}amiltonian systems on graphs.
\newblock {\em SIAM J. Control Optim.}, 51(2):906--937, 2013.

\bibitem{VvdS10a}
A.~Venkatraman and A.~van~der Schaft.
\newblock Energy shaping of port-hamiltonian systems by using alternate passive
  input-output pairs.
\newblock {\em Eur. J. Control}, 16(6):665--677, 2010.

\bibitem{VvdS10b}
A.~Venkatraman and A.~van~der Schaft.
\newblock Interconnections of port-{H}amiltonian systems: generating new
  passive outputs and feedback stabilization.
\newblock {\em IFAC Proceedings Volumes}, 43(14):605--610, 2010.

\bibitem{Wil10}
J.C. Willems.
\newblock Terminals and ports.
\newblock {\em IEEE Circuits and Systems Magazine}, 10(4):8--16, 2010.

\end{thebibliography}

\end{document}